\documentclass[a4paper,10pt]{article}
\pdfoutput=1
\usepackage{subfigure}
\usepackage{color}
\usepackage[pdftex]{graphicx}
\usepackage{fullpage}
\usepackage{hyperref}
\usepackage{mathrsfs}
\usepackage{amsmath,amssymb,latexsym,amsthm}

\theoremstyle{plain}
\newtheorem{thm}{Theorem}
\numberwithin{thm}{section}
\newtheorem{lem}[thm]{Lemma}
\newtheorem{defi}[thm]{Definition}
\newtheorem{prop}[thm]{Proposition}
\newtheorem{cor}[thm]{Corollary}

\theoremstyle{definition}

\numberwithin{defn}{section}

\theoremstyle{definition}
\newtheorem{rem}{Remark}
\numberwithin{rem}{section}

\numberwithin{equation}{section}

\numberwithin{equation}{section}
\def\N{\mathbb{N}}

\def\R{\mathbb{R}}

\def\Ll{L_{\textnormal{loc}}}
\def\Lip{\mathcal Lip}

\def\Supp{\mathop{\mbox{\textup{Supp}}}}
\def\div{\mbox{\textup{div}}}

\def\ep{\varepsilon}

\def\longrightharpoonup{\mathop{\relbar\joinrel\rightharpoonup}}
\title{The Vlasov-Navier-Stokes system in a 2D pipe: \\existence and stability of regular equilibria}

\author{Olivier Glass\footnote{CEREMADE, Universit\'e Paris-Dauphine, CNRS UMR 7534, PSL Research University, Place du Mar\'echal de Lattre de Tassigny, 75775 Paris Cedex 16, France},  Daniel Han-Kwan\footnote{CMLS - \'Ecole polytechnique, CNRS, Universit\'e Paris-Saclay, 91128 Palaiseau Cedex, France}\, and Ayman Moussa\footnote{Sorbonne Universit\'es, UPMC Univ Paris 06 \& CNRS, UMR 7598 LJLL, Paris, F-75005, France}
}
\begin{document}
\maketitle

%
%
\begin{abstract}
In this paper, we study the Vlasov-Navier-Stokes system in a 2D pipe with partially absorbing boundary conditions. We show the existence of stationary states for this system near small Poiseuille flows for the fluid phase, for which the kinetic phase is not trivial. We prove the asymptotic stability of these states with respect to appropriately compactly supported perturbations. The analysis relies on geometric control conditions which help to avoid any concentration phenomenon for the kinetic phase.
\end{abstract}
\section{Introduction}
Fluid-kinetic systems aim at describing the motion of a dispersed phase of particles within a fluid. The dispersed phase is represented by a density function solving a kinetic equation, whence the naming ``fluid-kinetic''. Historically these models were introduced in the combustion theory framework in the seminal thesis of O'Rourke \cite{oro}. The large amount of modelling possibilities (see \cite{wil} or \cite{duf_these,des-10}) for the fluid (compressibility, viscosity \emph{etc.}), the dispersed phase (thin/thick spray, Brownian motion \emph{etc.}) and their interaction (drag force, Basset force, lift force, retroaction \emph{etc.})  led to a constellation of fluid-kinetic couplings which have been studied from the mathematical point of view quite intensively in the past two decades. Chronologically the first mathematical studies focused more on the corresponding Cauchy problem and appeared in the late nineties, see for instance the works of Hamdache \cite{Hamdache} or Anoschchenko and Boutet de Monvel \cite{ano-bou}. In continuation of these papers, existence of weak solutions (see \cite{BDGM,mel-vas07a} and the more recent \cite{bou-gran-mou}) or classical solutions under smallness conditions, together with their long-time behavior (see \cite{Jabin,gou-he-mou-zha,cha-kan-lee-13,li-mu-wan-15}) or blow-up (see \cite{choi}) were explored. Another interesting feature of these systems is their link, through an asymptotic regime, to other physically relevant systems. These kinds of limits can be related to Hilbert's $6$th problem of axiomatization of physics, because they allow to derive rigorously the equation of continua from more elementary systems. Typical examples are \emph{hydrodynamic limits} for which the purpose is to replace the density function by averaged quantities (mass, momentum \emph{etc.}) in order to recover, after a rigorous asymptotic, classical equations of fluid mechanics, see  \cite{GJV1}, \cite{ben-des-mou} or \cite{dgrb} for an example involving a collision operator. Since kinetic equations are not ``first principles'' \emph{per se}, a comprehensive derivation would suggest to obtain asymptotically fluid-kinetic systems starting from fluid-solid equations. These \emph{mean-field limits} have been explored recently \cite{dgr,hill} using homogenization techniques reminiscent of the pioneering works of Allaire, Murat and Cioranescu \cite{allaire,strange}. Finally, in an another direction, Moyano \cite{Moy1,Moy2} recently studied the controllability properties of fluid-kinetic systems.

\bigskip

Our paper deals with the long time behavior of the Vlasov-Navier-Stokes system. 
Its originality with respect to the previous state of the art stems from the fact that we consider solutions around a nontrivial stationary solution. As it will be explained below, the very existence of such an equilibrium is already remarkable and relies strongly on the geometry of the domain and the boundary conditions that we enforce. Up to our knowledge, as far as fluid-kinetic couplings are concerned, stability of stationary solution was up to now especially tackled when the kinetic equations includes a Fokker-Planck term, allowing to consider the equilibrium $(u=0,f=\mathcal{M})$ (as usual $u$ stands for the velocity of the fluid, and $f$ for the distribution function of the kinetic phase), where $\mathcal{M}(v):=e^{-|v|^2/2}$ is the standard Maxwellian, around which smooth solutions can be studied (see \cite{gou-he-mou-zha,cardumou} and the references therein). When the dispersed phase is submitted to a drag force, if no diffusive term smoothes out the kinetic equation, and in the absence of any dissipative mechanism, the only nontrivial equilibria that one can imagine are singular, in the sense that the density function becomes monokinetic (that is to say, Dirac measures in velocity), see the discussion below. The purpose of this paper is to find for the Vlasov-Navier-Stokes coupling a  relevant setting for which nontrivial stationary solutions exist and study the local asymptotic stability of such equilibria. We consider, in dimension $2$, the Vlasov-Navier-Stokes system:
\begin{gather}
\label{Kinetic}
\partial_{t} f + v \cdot \nabla_{x} f + \div_{v} ((u-v) f) = 0 \text{ for } (t,x,v) \text{ in } \R_+ \times \Omega \times \R^{2}, \\
\label{NavierStokes1}
\partial_{t} u + (u \cdot \nabla) u - \nu \Delta u + \nabla p = \int_{\R^{2}} f(t,x,v) (v-u(t,x)) \, dv  \text{ for } (t,x) \text{ in } \R_+ \times \Omega , \\
\label{NavierStokes2}
\div\, u=0 \text{ for } (t,x) \text{ in } (0,T) \times \Omega,
\end{gather}
under relevant boundary conditions, where $\Omega$ is a domain of $\R^2$ which we shall describe later. This system couples a Vlasov equation with friction (drag force) to the incompressible Navier-Stokes equations through a forcing term. In these equations, $u=u(t,x)$ represents the velocity field of the fluid phase at time $t$ and position $x \in \Omega$, while $f=f(t,x,v)$ stands for the distribution function of a repartition of some particles in the phase space $\Omega \times \R^2$. 

\par

\bigskip
Before specifying precisely $\Omega$ and the boundary conditions, we begin with some remarks about the large time behavior of solutions in the case $\Omega= \R^2$ (taking for granted the existence of solutions, that we shall not discuss for now). Forgetting for a while about the Navier-Stokes equations, we therefore first consider the Vlasov equation with friction
$$
\partial_{t} f + v \cdot \nabla_{x} f - \div_{v} (v f) = 0 \ \text{ in } \ (0,T) \times \R^{2} \times \R^{2},
$$
whose solutions are explicitly given by
$$
f(t,x,v) = e^{2t} f_{|t=0} (x- (e^{t}-1) v, e^{t} v).
$$
One deduces that their behavior as $t \to +\infty$ is as follows:
\begin{equation*}
\displaystyle f(t,x,v) \longrightharpoonup_{t \to +\infty} \left( \int_{\R^2} f_{|t=0} (x-v,v) \, dv \right) \delta_{v=0},
\end{equation*}
where $\delta$ stands for the Dirac measure. Two remarks are in order:
\begin{itemize}
\item the trivial distribution function $0$ is the only $L^1_{\textnormal{loc}}$ stationary solution of the Vlasov equation with friction;
\item this trivial distribution is unstable as any non-zero initial condition yields a solution that weakly converges to a Dirac mass as time goes to infinity.
\end{itemize}
Considering again the coupling with Navier-Stokes,  one may expect a similar large time behavior for the distribution function. Indeed, some evidence is brought by the fact that the linearized equations around the trivial state $(u=0, f=0)$ are the decoupled Vlasov with friction and Stokes equations (without source). Therefore, at least in the small data regime, we do not expect the Navier-Stokes to prevent the convergence of the distribution function to Dirac measures. However, no result in this direction has been proved rigorously, at least to the best of our knowledge. 
Let us mention though some partial results: in \cite{choif} Choi and Kwon managed to exhibit a monokinetic behavior for the Vlasov-Navier-Stokes but only under the assumption of strong \emph{a priori} estimates; in \cite{Jabin}, Jabin replaces the Navier-Stokes by a stationary Stokes equation (also with a different coupling term) and was able to prove  
that there is indeed also convergence to a Dirac measure in $0$.

The mechanism at stake in the Vlasov equation with friction on the whole space is a competition between \emph{dispersion} and \emph{friction}, with friction always taking over in the end. We shall see that in a domain $\Omega$ with (partially) absorbing boundary conditions, there are some cases where dispersion can be the dominant effect, because particles will be absorbed before the friction mechanism can become efficient. \par
\bigskip
The (simple) geometry that we consider in this paper is as follows. The system is posed in a pipe
\begin{equation*}
\Omega := (-L,L) \times (-1,1),
\end{equation*}
with $L>0$. Let us consider the vertical/horizontal incoming boundaries of the phase space domain (with the four corners excluded):
\begin{gather*}
\Gamma^{l} := \{- L\} \times (-1,1) \times \{ v_1 >0\} \ \text{ and } \ 
\Gamma^{r} := \{ L\} \times (-1,1) \times \{v_1 <0\}, \\
\Gamma^{u} := (-L,L) \times \{1\} \times \{v_2 <0\}  \ \text{ and } \ 
\Gamma^{d} := (-L,L) \times \{- 1\} \times \{ v_2 >0\},
\end{gather*}
in such a way that $\Sigma^-:=\Gamma^{l} \cup \Gamma^{r} \cup \Gamma^{u} \cup \Gamma^{d}$ gather all (non corners) phase space points $(x,v)$ such that $ n(x) \cdot v <0 $, where $n(x)$ is the outward unit vector defined for $x$ in $\partial\Omega \setminus\{(\pm L,\pm 1)\}$. \par
\begin{figure}[htbp]
	\begin{center}
		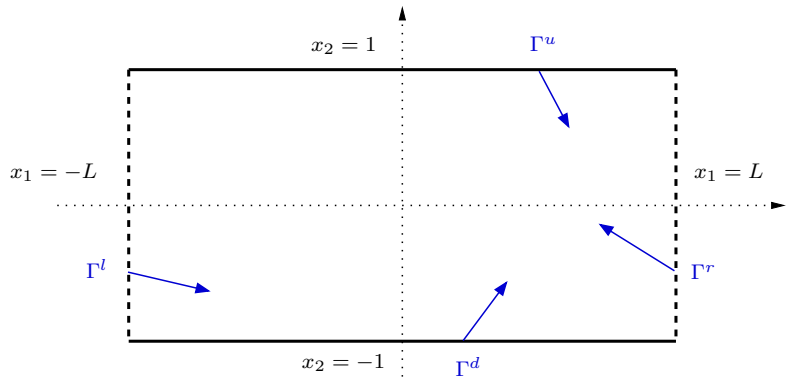
	\end{center}
	\caption{The domain $\Omega$}
	\label{fig:domain}
\end{figure}
Likewise, we define the outgoing boundary of the phase space domain as $\Sigma^+:=\{(x,v)\,:\, (x,-v) \in \Sigma^-\}$ and eventually we introduce the singular subset 
\begin{equation} \label{Eq:DefSigmas}
\Sigma^s := (\partial \Omega \times \R^2) \setminus (\Sigma^+ \cup \Sigma^-), 
\end{equation}
gathering all the boundary phase space points for which either the normal is not defined (\emph{i.e.} corners) or such that $v$ is tangent to $\partial\Omega$. We enforce the following mixed incoming/absorbing and Dirichlet boundary conditions for $t$ in $\R_+$:
\begin{gather}
 \label{CBu1}
u(t,x) =0 \text{ for } x \text{ on } (-L,L) \times \{-1,1\}, \\
 \label{CBu2}
u(t,x) = u_p(x) \text{ for } x \text{ on }  \{-L,L\} \times [-1,1], \\
 \label{CBf1}
f(t,x,v) = \psi(x_2,v) \text{ for } (x,v) \text{ in } \Gamma^l, \\
\label{CBf2}
f = 0 \text{ for } (x,v) \text{ on } \Gamma^r \cup \Gamma^u \cup \Gamma^d,
\end{gather}
where $u_p(x)$ is a given \emph{Poiseuille flow}  (that we introduce below) and $\psi(x_2,v)$ is a given \emph{incoming} distribution function, defined on $\Gamma^l$. In this paper $\psi$ will be assumed to be nonnegative and to have a compact support in $(-1,1) \times \R^{2}$. We underline that $\psi$ being compactly supported in space in $(-1,1)$, compatibility conditions at $x=(-L,1)$ and $x=(-L,-1)$ are satisfied for free. 
Note that the boundary conditions~\eqref{CBf1} and~\eqref{CBf2} physically mean that some particles are injected into the pipe on the side $x_1=-L$ (and only on this side), while they are absorbed when reaching transversally the sides $x_2=-1,1$ and $x_1=-L,L$. The Vlasov-Navier-Stokes system in this particular geometry and with these boundary conditions is the two-dimensional version of a model used to describe the transport and deposition of aerosol inside the human upper airways, see \emph{e.g.} \cite{bou-gra-lor-mou} for more details concerning modelling issues.
\par 
\bigskip
We recall that a Poiseuille flow is a particular stationary solution of \eqref{NavierStokes1}-\eqref{NavierStokes2} without source, that is
\begin{gather}
\label{SNS1}
(u \cdot \nabla) u - \nu \Delta u + \nabla p = 0 \text{ in }  \Omega , \\
\label{SNS2}
\div\, u=0 \text{ in } \Omega ,
\end{gather}
and is explicitly given by
\begin{equation} \label{Poiseuille}
u_{p}(x) = (1-x_2^{2}) u_{\rm max} e_{1},
\end{equation}
where we denote by $(e_{1},e_{2})$ the canonical basis of $\R^{2}$, with associated pressure $p= -2  \nu u_{\rm max} x_1$.
We call the corresponding $u_{\rm max}>0$ the {\it intensity} of the Poiseuille flow. Note also that $u_p$ satisfies the boundary condition~\eqref{CBu1}, so that \eqref{CBu1} and \eqref{CBu2} may be merged into the single boundary condition
\begin{equation*}
u(t,x) =u_p(x) \text{ for } x \text{ on } (-L,L) \times \{-1,1\} \cup  \{-L,L\} \times [-1,1].
\end{equation*}
\begin{rem}
\label{BC0}
The homogeneous Dirichlet boundary condition~\eqref{CBu1} is taken with a modelling perspective, as it may be relevant for a human lung.  As a matter of fact, that $u$ is identically $0$ on $(-L,L) \times \{-1,1\}$ is at the origin of several complications in the subsequent analysis. Other choices could have been possible; for instance we could have taken instead of~\eqref{CBu1}-\eqref{CBu2}:
\begin{equation} \label{Poiautre}
u(t,x) =u'_p(x) \text{ for } x \text{ on } (-L,L) \times \{-1,1\} \cup  \{-L,L\} \times [-1,1],
\end{equation}
for $u'_{p}(x) = (1-\lambda x_2^{2}) u_{\rm max} e_{1}$ with $\lambda \in (0,1)$. In that case observe that the modulus of $u'_p$ is uniformly bounded from below on $\overline \Omega \times \R^2$. With such a choice, the analysis of this paper could still be carried out but would be quite simplified and the result would be strenghtened because of this uniform bound from below. We refer to Remarks~\ref{BC1} and~\ref{BC2}.
\end{rem}
%
We shall refer to the system \eqref{Kinetic}-\eqref{NavierStokes1}-\eqref{NavierStokes2}-\eqref{CBu1}-\eqref{CBu2}-\eqref{CBf1}-\eqref{CBf2} as the Vlasov-Navier-Stokes system in the pipe.
The  goal of this work is to study small data solutions whose fluid velocity fields are close to a Poiseuille flow. %
\par
\bigskip
Let us gather the main statements of this paper in the following informal theorem. We will make them  more precise in the course of the paper.
%
%
\begin{thm}

\begin{itemize}
\item For a sufficiently small Poiseuille flow $u_p$ (as boundary condition for the fluid equation), there exist boundary data $\psi$ for the kinetic equation such that the corresponding Vlasov-Navier-Stokes system admits nontrivial stationary solutions $(\overline{u},\overline{f})$ close to $(u_p,0)$. By nontrivial, we mean that the kinetic part is not identically $0$. (See Theorem~\ref{Theo:Equilibria} for a more precise statement.)
\item The nontrivial stationary solutions $(\overline{u},\overline{f})$ introduced above satisfy an exponential stability property, in the sense that any Leray solution of the Vlasov-Navier-Stokes system with the same boundary conditions as $(\overline{u},\overline{f})$ and appropriately compactly supported initial condition sufficiently close to $(\overline{u},\overline{f})$ converges exponentially fast to $(\overline{u},\overline{f})$ as the time goes to infinity. (See Theorem~\ref{thm-expo} for a more precise statement.)

\item A consequence is a local uniqueness property of the aforementioned stationary solutions (see Corollary \ref{cor:uni}).

\end{itemize}
\end{thm}
Loosely speaking, we therefore prove existence and stability of  nontrivial regular stationary states for the Vlasov-Navier-Stokes system. This follows from key geometric conditions (that we will refer to as \emph{exit geometric conditions}) satisfied by the Poiseuille flow. Roughly speaking, we ask that all associated characteristics emerging from the support of the incoming data $\psi$ leave transversally the domain before a fixed time. 
As some preparation is needed, we have chosen to postpone the statement of the exact definitions.
These exit geometric conditions are somehow reminiscent of the celebrated geometric control condition of Bardos, Lebeau and Rauch \cite{BLR}\footnote{This Geometric Control Condition was introduced in the context of the controllability of the wave equation.}. Also, the term ``appropriately compactly supported'' in the statement means that the compact support of the perturbation has to verify some geometric assumption that will be precisely stated in the course of the paper.
Finally the local uniqueness property concerns a weighted $L^\infty$/Lipschitz space.
\ \par
\bigskip
The structure of this paper is as follows. First, in Section~\ref{Sec:Cauchy} we study the Cauchy problem for the Vlasov-Navier-Stokes system in the pipe. 
More precisely we develop a theory of Leray solutions in this geometry (including several regularity  estimates). 
Next in Section~\ref{Sec:EGC}, we introduce the geometric exit conditions mentioned above and establish some continuity properties related to those. Then, in Section~\ref{Sec:Equilibria} we show the existence of nontrivial stationary solutions of the Vlasov-Navier-Stokes system in the pipe 
with small data.
This first result already confirms that the situation is very different for the case of the whole space. We note that the stationary solutions that we build satisfy the aforementioned exit geometric condition. 
In Section~\ref{Sec:Stab}, we prove the exponential stability of the stationary solutions introduced in Section~\ref{Sec:Equilibria}, provided again that the reference Poiseuille flow is sufficiently small.  As already said, the analysis crucially relies on the exit geometric condition. In the course of the proof, we will use a Gronwall-type lemma for a kind of delayed differential inequality. Several regularity results for the Navier-Stokes equation in 2D are also needed. The local uniqueness property of the stationary solutions built in Section~\ref{Sec:Equilibria} will be a direct consequence of this stability. \par
Finally Section~\ref{Sec:Appendix} is an Appendix where we gather various technical results needed in the proofs.

\vspace{2mm}

\noindent \textbf{Acknowledgments.}
O.G. was partially funded by the French ANR-13-BS01-0003-01 Project DYFICOLTI, D.H.-K. by a PEPS-JC granted by CNRS and A.M. by the French ANR-13-BS01-0004 project KIBORD.

\section{The Cauchy problem for the Vlasov-Navier-Stokes system in the pipe} 
\label{Sec:Cauchy}
For any distribution function $f:\R_+ \times \Omega \times \R^2\rightarrow\R_+$ we define, for all $\alpha \geq 0$, the moments
\begin{align*}
m_\alpha f (s,x) &:= \int_{\R^2} f(s,x,v)|v|^\alpha \, dv,\\
M_\alpha f(s) &:= \int_{\Omega} m_\alpha f(s,x) \, dx, \\
j_f(s,x) &:= \int_{\R^2} f(s,x,v)v  \, dv.
\end{align*}
%
%
We will consider weak solutions of the Vlasov equation~\eqref{Kinetic}, with the boundary conditions \eqref{CBf1}-\eqref{CBf2}, which are precisely defined as follows.
\begin{defi}[Weak solutions for the Vlasov equation]
\label{def:weak}
Given $\psi\in L^\infty(\R_+ \times \Gamma^l)$ an entering distribution, a vector field $u \in L^1(\R_+\times \Omega)$,  we say that a distribution function $f\in L^\infty(\R_+ \times \Omega \times \R^2) \cap \mathscr{C}^0(\R_+;L^\infty(\Omega\times\R^2)-w\star)$ is a weak solution of \eqref{Kinetic} together with boundary conditions \eqref{CBf1}-\eqref{CBf2} and initial data  $f_0 \in L^\infty\cap L^1(\Omega\times\R^2)$ if, for all $\phi\in\mathscr{D}(\R_+ \times \overline{\Omega} \times\R^2)$ vanishing on $\Sigma^+$, one has, for all $T>0$,
\begin{multline*}
\int_0^T \! \int_\Omega \! \int_{\R^2} f[\partial_t \phi + v\cdot \nabla_x \phi + (u-v) \cdot \nabla_v \phi](t,x,v) \, dv \, dx \, dt\\
= \int_\Omega \! \int_{\R^2} f(T,x,v)\phi(T,x,v) \, dv \, dx - \int_\Omega \! \int_{\R^2} f_0(x,v)\phi(0,x,v) \, dv \, dx  \\
+ \int_0^T \! \int_{\Gamma^l} \psi(t,x_2,v) \phi(t,0,x_2,v) v_1 \, dv \, dx_2 \, dt.
\end{multline*}
\end{defi}
%
%
%
We will consider Leray solutions of the full Vlasov-Navier-Stokes system \eqref{Kinetic}-\eqref{NavierStokes2}, with boundary conditions \eqref{CBu1}-\eqref{CBf2}, in the following sense.
\begin{defi}[Leray solutions for the Vlasov-Navier-Stokes system]
\label{def:leray} 
Given $\psi\in L^\infty(\R_+\times\Gamma^l)$ a nonnegative entering distribution with compact support in velocity included in $B(0,R)$,  $u_0\in L^2(\Omega)$ with $\div \, u_0=0$ and $f_0\in L^\infty \cap L^1 (\Omega\times\R^2)$ a nonnegative function such that $M_4 f_0 <\infty$, we say that the couple $(u,f)$ is a Leray solution to the the Vlasov-Navier-Stokes system \eqref{Kinetic}-\eqref{NavierStokes2}, with boundary conditions \eqref{CBu1}-\eqref{CBf2} and initial conditions $(u_0, f_0)$ if
$$ u-u_p \in \mathscr{C}^0(\R_+; L^2(\Omega))\cap L^2_{\textnormal{loc}}(\R_+; H^1_0(\Omega)), $$
$$ f \in \mathscr{C}^0(\R_+; L^\infty(\Omega \times \R^2)-w\star) \cap \mathscr{C}^0(\R_+; L^p(\Omega\times\R^2)) \text{ for any } p<\infty, $$ 
and the following holds. 
The distribution function $f$ is a nonnegative weak solution of the Vlasov equation with force field $u$ in the sense of Definition~\ref{def:weak}
and for all $\Phi \in \mathscr{C}^1(\R_+; H^1_0(\R^2))$ such that $\div \, \Phi = 0$, we have, for $T\geq 0$,
\begin{multline*}
\int_0^T \! \int_\Omega [u \cdot \partial_t \Phi + u \otimes u \, : \, \nabla_x \Phi - \nabla_x u \, : \, \nabla_x \Phi ](t,x) \, dx \, dt \\
= \int_\Omega u(T,x) \cdot \Phi(t,x) dx-\int_\Omega u_0(x) \cdot \Phi(0,x) \, dx \\
- \int_0^T \! \int_\Omega  \int_{\R^{2}} f(t,x,v) (v-u(t,x)) \cdot \Phi(t,x) \, dv \, dx \, dt.
\end{multline*}
Furthermore, denoting 
\begin{equation} \label{DefUTilde}
\widetilde{u}:=u-u_p,
\end{equation}
one has the following energy estimate for all $T \geq 0$ and all $t\in [0,T]$,
\begin{multline} \label{ineq:nrj-def}
\| \widetilde{u}(t) \|_2^2 + M_2 f(t) + \int_0^{t} \| \nabla \widetilde{u}(s) \|_2^2 \, ds 
+ \int_0^{t}\int_{\Omega\times\R^2} f(s,x,v) |u(s,x)-v|^2 \, dv \, dx \, ds \\
\leq C_{\Omega,T,R,u_p}(\|f_0\|_\infty, M_2f_0,\|u_0-u_p\|_2,\|\psi\|_\infty),
\end{multline}
in which $C_{\Omega,T,R,u_p}(\cdot,\cdot,\cdot,\cdot)$ stands for a positive continuous functions that is nondecreasing with respect to each of its arguments.
\end{defi}
The constant $C_{\Omega,T,R,u_p}$ above can actually be made explicit in the proof of existence of such Leray solutions.
\begin{rem}
Notice that we prove here the existence of Leray solutions in the case when $\psi$ depends on the time variable, which is a slight generalization of our original setting. 
\end{rem}
\begin{rem}
The assumptions $f_0 \in L^\infty(\Omega\times \R^2)$ and $M_4 f_0 <\infty$ imply, thanks to the interpolation Lemma~\ref{interpol}, that $M_2 f_0 < \infty$. The assumption of compact support in velocity for the entering distribution is not mandatory here and could be replaced by $M_1 \psi +M_4 \psi \in L^\infty(\R_+)$ for instance.
\end{rem}
\begin{rem}
The Poiseuille flow $u_p$ is included in the energy estimate \eqref{ineq:nrj-def} since it is used as a lifting to handle the non homogeneous boundary conditions.
\end{rem}
The following result gives additional estimates satisfied by Leray solutions.
\begin{prop} \label{propo:nrj-def}
For any $T> 0$, any Leray solution $(u,f)$ in the sense of Definition~\ref{def:leray} is actually such that $u\in L^2_{\textnormal{loc}}(\R_+^*;H^2(\Omega))\cap\mathscr{C}^0(\R_+^*; H^1(\Omega))\cap \mathscr{C}^1(\R_+^* ; L^2(\Omega)) $ and satisfies the following additional estimates on any interval $(a,b)=(a,a+T) \subset \R_{+}^*$ (denoting again $\widetilde{u}=u-u_p$),
\begin{itemize}
\item Shifted energy inequality: for all $t\in [a,b]$, we have
\begin{multline} \label{ineq:nrj}
\| \widetilde{u}(t) \|_2^2 + M_2 f(t) + \int_a^{t} \| \nabla \widetilde{u}(s) \|_2^2 \, ds 
+ \int_a^{t} \! \int_{\Omega\times\R^2} f(s,x,v)|u(s,x)-v|^2 \, dv \, dx \, ds \\
\leq C_{\Omega,T,R,u_p}(\| f(a) \|_\infty, M_2f(a), \| \widetilde{u}(a) \|_2, \| \psi \|_\infty).
\end{multline} 
\item Maximum principle for the distribution function: we have
\begin{align} \label{ineq:max}
\sup_{a \leq s \leq b}\|f(s)\|_\infty &\leq e^{2T}(\|f(a)\|_\infty+ \|\psi\|_\infty).
\end{align}
\item Propagation of moments for the distribution function: we have $M_4f(a) <\infty$ and 
\begin{align} \label{ineq:moments}
\sup_{a \leq s \leq b}M_4 f(s)  &\leq C_{\Omega,T,R}(M_4 f(a),\|f(a)\|_\infty, \|\psi\|_\infty)  D_{\Omega,T,R,u_p}(\|\widetilde{u}(a)\|_2).
\end{align}
\item Regularity estimate: if $c:=\max(a-T/2,0)$,
\begin{align} \label{ineq:ell}
\int_a^b \|\widetilde{u}(s)\|_\infty \, ds \leq C_{\Omega,T,R,u_p}(\|f(c)\|_\infty,M_4 f(c),\|\widetilde{u}(c)\|_2, \|\psi\|_\infty).
\end{align}
\end{itemize}
In all of these estimates, $D_{\Omega,T,R,u_p}$, $C_{T,\Omega,R}$  and $C_{T,\Omega,R,u_p}$ denote generic positive continuous functions  (that may vary from line to line in the proof below) nondecreasing with respect to each of their arguments, the second and third one furthermore vanishing at $0$.
\end{prop}
\begin{rem}
In view of the upcoming stability part, it is crucial to note that the constants $C_{\Omega,T,R}$ and $D_{\Omega,T,R,u_p}$ do not depend on $a,b$ but only on $T=b-a$.
\end{rem}
\begin{proof}[Proof of Proposition~\ref{propo:nrj-def}]
The proof will use several results gathered in the Appendix (Section~\ref{Sec:Appendix}).
Let us first notice that since $(u,f)$ is a Leray solution, we have $u \in L^2_{\textnormal{loc}}(\R_+;L^6(\Omega))$ by the Sobolev embedding $H^1(\Omega)\hookrightarrow L^6(\Omega)$, so that Corollary~\ref{coro:vla} applies and we have in particular $M_4 f \in L^\infty_{\textnormal{loc}}(\R_+)$ because $M_4 f_0 <\infty$.
Since $f\in L^\infty_{\textnormal{loc}}(\R_+; L^\infty(\Omega\times \R^2))$, using Lemma~\ref{interpol} we get $m_1 f \in L^\infty_{\textnormal{loc}}(\R_+; L^2(\Omega))$ and $m_0 f \in L^\infty_{\textnormal{loc}}(\R_+;L^3(\Omega))$. More precisely if $b=a+T$, we have
\begin{align} \label{ineq:sourcel2}
\|j_f - (m_0 f)u\|_{L^2(a,b;L^2(\Omega))} \leq C_\Omega(\|M_4 f\|_{L^\infty(a,b)},\|f\|_{L^\infty(a,b;L^\infty(\Omega\times\R^2))}) ( 1+ \|u\|_{L^2(a,b;L^6(\Omega))}),
\end{align}
for some continuous function $C_\Omega$ vanishing at $0$ (increasing with respect to both of its arguments).
This means in particular that $u$ solves the Navier-Stokes equation on $\R_+\times\Omega$ with a source term in $L^2_{\textnormal{loc}}(\R_+;L^2(\Omega))$, so that the parabolic regularization over time satisfied by this equation (and stated in  Theorem~\ref{thm:parabo}) applies and we indeed get $u\in L^2_{\textnormal{loc}}(\R_+^*;H^2(\Omega))\cap \mathscr{C}^0(\R_+^*;H^1(\Omega))\cap \mathscr{C}^1(\R_+^*;L^2(\Omega))$. \par
\ \par
\noindent $\bullet$ Now let us prove \eqref{ineq:nrj} for $a>0$ (the case $a=0$ is in fact a consequence of Definition~\ref{def:leray}). For any $t\in[a,b]=[a,a+T]$, taking the difference of the two weak formulations of Definition~\ref{def:leray} at time $t$ and time $a$, we get 
\begin{multline*}
\int_a^t \! \int_\Omega [u \cdot \partial_t \Phi + u \otimes u \, : \, \nabla_x \Phi - \nabla_x u \, : \, \nabla_x \Phi ](s,x) \, dx \, ds \\
= \int_\Omega u(t,x) \cdot \Phi(t,x) dx-\int_\Omega u(a,x) \cdot \Phi(a,x) \, dx \\
- \int_a^t \! \int_\Omega  \int_{\R^{2}} f(s,x,v) (v-u(s,x)) \cdot \Phi(s,x) \, dv \, dx \, ds.
\end{multline*}
Since $\widetilde{u}:=u-u_p$ belongs to $\mathscr{C}^1(\R_+^*;L^2(\Omega))\cap L^2_{\textnormal{loc}}(\R_+;H^1_0(\Omega))$ a straightforward density argument implies 
\begin{multline*}
\int_a^t \! \int_\Omega [u \cdot \partial_t \widetilde{u} + u \otimes u \, : \, \nabla_x \widetilde{u} - \nabla_x u \, : \, \nabla_x \widetilde{u} ](s,x) \, dx \, ds \\
= \int_\Omega u(t,x) \cdot \widetilde{u}(t,x) dx-\int_\Omega u(a,x) \cdot \widetilde{u}(a,x) \, dx \\
- \int_a^t \! \int_\Omega  \int_{\R^{2}} f(s,x,v) (v-u(s,x)) \cdot \widetilde{u}(s,x) \, dv \, dx \, ds.
\end{multline*}
Since $u_p$ is a stationary solution of Navier-Stokes, 
we infer after some integration by parts (using $\text{div}\,u = 0$)
\begin{multline}  \label{eq:nrjflu}
\frac{1}{2} \|\widetilde{u}(t)\|_2^2 + \int_a^t \|\nabla \widetilde{u}(s)\|_2^2 \, ds
= \int_a^t \! \int_{\Omega} (j_f -(m_0 f)u )(s,x)\cdot \widetilde{u}(s,x) \, dx \, ds \\
-\int_a^t \! \int_\Omega \big[(\widetilde{u} \cdot\nabla) u_p \big]\cdot \widetilde{u}(s,x) \,dx \, ds+ \frac{1}{2}\|\widetilde{u}(a)\|_2^2.
\end{multline}
Now it may be checked  that $(t,x,v)\mapsto f(a+t,x,v)$ is solution of \eqref{Kinetic} in the sense of Definition~\ref{def:weak} with initial condition $f(a)\in L^\infty(\Omega\times\R^2)$ (notice that $M_2 f(a) <\infty$) and with the vector field $(t,x)\mapsto u(a+t,x)$. We may invoke the moments estimate of Theorem~\ref{thm:vla} (with $\chi(z)=z$) to get
\begin{equation} \label{eq:momentkin1}
M_0 f (t) = M_0 f(a) - \int_a^t \int_{\Gamma^l} \psi(s,x,v) v\cdot n(x) \,dv\,dx\,ds,
\end{equation}
and (taking $\alpha=2$)
\begin{multline} \label{eq:momentkin2}
\frac{1}{2}M_2 f(t)  + \frac{1}{2} \int_a^t M_2f(s) \, ds  = \int_a^t \int_{\Omega\times\R^2} j_f \cdot u(s,x) \, dx \, ds + \frac{1}{2}M_2 f_a \\
- \frac{1}{2}\int_a^t \int_{\Gamma^l} \psi(s,x)|v|^2 v\cdot n(x) \, dx \, dv.
\end{multline}
%
Summing  \eqref{eq:nrjflu} and \eqref{eq:momentkin2} we get for all $t\in[a,b]$
\begin{multline*}
E(f(t),u(t)) +\int_a^t \|\nabla \widetilde{u}(s)\|_2^2 \, ds + \frac{1}{2}\int_a^t \int_{\Omega\times\R^2} f(s,x,v)|u(s,x)-v|^2 \, dv \, dx \, ds \\
= \int_a^t \! \int_{\Omega} ((m_0 f) u -j_f)(s,x) \cdot u_p(x) \, dx \, ds 
 - \int_a^t \! \int_\Omega \big[(\widetilde{u} \cdot\nabla) u_p\big] \cdot \widetilde{u}(s,x) \, dx \, ds \\
 - \frac{1}{2} \int_a^t \! \int_{\Gamma^l} \psi(s,x,v)|v|^2 v\cdot n(x) \, dx \, dv
 + E(f(a),\widetilde{u}(a)),
\end{multline*}
where $2 E(f(t),u(t)) = \|\widetilde{u}(t)\|_2^2 + M_2 f(t)$. Recalling that  $\text{Supp}_v(\psi)\subset B(0,R)$ and that $u_p$ is Lipschitz, we obtain
\begin{multline*}
E(f(t),u(t)) +\frac{1}{2}\int_a^t \|\nabla \widetilde{u}(s)\|_2^2 \, ds 
+ \frac{1}{2} \int_a^t \int_{\Omega\times\R^2} f(s,x,v) |u(s,x)-v|^2 \, dv \, dx \, ds \\
\leq C_{\Omega,T,R}(\|\psi\|_\infty) + E(f(a),\widetilde{u}(a)) 
+ \|\nabla u_p\|_\infty  \int_a^t \| \widetilde{u} \|_{2}^2 \, ds \\
+ \int_a^t \int_{\Omega} ((m_0 f) u -j_f)(s,x) \cdot u_p(x) \, dx \, ds.
\end{multline*}
Now using Young's inequality together with the nonnegativity of $f$, we have 
\begin{equation*}
f(u-v)  \cdot u_p  \leq \frac{1}{4} f|u-v|^2  + {2}f |u_p|^2,
\end{equation*} 
so that
\begin{multline*}
E(f(t),u(t)) + \frac{1}{2} \int_a^t \|\nabla \widetilde{u}(s)\|_2^2 \, ds 
+ \frac{1}{4} \int_a^t \! \int_{\Omega\times\R^2} f(s,x,v)|u(s,x)-v|^2 \, dv \, dx \, ds \\
\leq C_{\Omega,T,R}(\|\psi\|_\infty) + E(f(a),\widetilde{u}(a)) + \|\nabla u_p\|_\infty\int_a^t \| \widetilde{u} \|_{2}^2 \,ds \\
+ {2} \|u_p\|_\infty \int_a^t \! \int_\Omega m_0 f(s,x) \, dx \, ds.
\end{multline*}
But using \eqref{eq:momentkin1} we get
\begin{equation*}
M_0f(t) \, = M_0 f(a)  -  \int_a^t \int_{\Gamma^l} \psi(s,x,v) v\cdot n(x) \, dx  \, dv,
\end{equation*}
and since $M_0 f(a) \leq C \|f(a)\|_\infty + M_2 f(a)$, changing the definition of $C_{\Omega,T,R}$ we may write 
\begin{multline*}
E(f(t),u(t)) + \frac{1}{2} \int_a^t \|\nabla \widetilde{u}(s)\|_2^2 \, ds 
+ \frac{1}{4}\int_a^t \int_{\Omega\times\R^2} f(s,x,v)|u(s,x)-v|^2 \, dv \, dx \, ds \\
\leq C_{\Omega,T,R,u_p}(\|\psi\|_\infty,\|f(a)\|_\infty,M_2 f(a),\|\widetilde{u}(a)\|_2) 
+ \|\nabla u_p\|_\infty \int_a^t \| \widetilde{u} \|_{2}^2 \,ds,
\end{multline*}
so that \eqref{ineq:nrj} follows from Gronwall's lemma. \par
\ \par
\noindent $\bullet$ Recall that $(t,x,v)\mapsto f(a+t,x,v)$ solves the Vlasov equation with field $(t,x)\mapsto u(a+t,x)$ and initial condition $f(a)$.
In particular, estimate \eqref{ineq:max} is a direct consequence of the maximum principle on $[0,T]$, and estimate of Corollary~\ref{coro:vla} rewrites here as (since $M_4 f(a) <\infty$)
\begin{equation*}
M_4 f(t) \leq C_{\Omega,T,R}(M_4 f(a),\|f(a)\|_\infty, \|\psi\|_\infty)  D_T(\|u\|_{L^2(a,b;L^6(\Omega))}),
\end{equation*}
so that using the estimate \eqref{ineq:nrj} that we have just proved (with the Sobolev embedding $H^1(\Omega)\hookrightarrow L^6(\Omega)$),
we recover \eqref{ineq:moments} noticing that $M_2 f(a) \leq C \|f(a)\|_\infty + M_4f(a)$. \par
\ \par
\noindent $\bullet$ It now remains to treat \eqref{ineq:ell}. Note that $\widetilde{u}$ solves 
\begin{equation*}
\partial_{t} \widetilde{u} - \Delta \widetilde{u} + \nabla \widetilde{p}
= - (u_p \cdot \nabla) \widetilde{u} - (\widetilde{u} \cdot \nabla) u_p + j_{f} - (m_0 f) u =:F,
\end{equation*}
where $\widetilde{p} = p-q$ with $q$ standing for the pressure associated to the Poiseuille flow $u_p$.
Using a time-translation argument, we infer from Theorem~\ref{thm:parabo} of the Appendix that
\begin{equation}
\label{eq-un}
\int_c^b (s-c) \|\widetilde{u}(s)\|_{H^2(\Omega)}^2 \, ds \leq C_{T,\Omega}\Big(\|u(a)\|_{L^2(\Omega)},\|F\|_{L^2(c,b;L^2(\Omega))}\Big).
\end{equation}
Let us assume for one moment that 
\begin{align}
\label{ineq:assF} \|F\|_{L^2(c,b;L^2(\Omega))} \leq C_{\Omega,T,R,u_p}(\|f(c)\|_\infty,M_4 f(c),\|\widetilde{u}(c)\|_2, \|\psi\|_\infty).
\end{align}
Then from estimate \eqref{ineq:nrj}, we infer a similar control for $\|u(a)\|_{L^2(\Omega)}$, so that from \eqref{eq-un} we get 
\begin{align*}
\int_c^b (s-c) \|\widetilde{u}(s)\|_{H^2(\Omega)}^2 \, ds \leq C_{\Omega,T,R,u_p}(\|f(c)\|_\infty,M_4 f(c),\|\widetilde{u}(c)\|_2, \|\psi\|_\infty).
\end{align*}
If $a>T/2$, that is $c=a-T/2$, estimate \eqref{ineq:ell} follows directly because 
\begin{align*}
\int_a^b \|\widetilde{u}(s)\|_{H^2(\Omega)}^2 ds \leq \frac{2}{T}\int_a^b(s-c) \|\widetilde{u}(s)\|_{H^2(\Omega)}^2 ds \leq  \frac{2}{T}C_{\Omega,T,R,u_p}(\|f(c)\|_\infty,M_4 f(c),\|\widetilde{u}(c)\|_2, \|\psi\|_\infty),
\end{align*}
and we get \eqref{ineq:ell} by the (two-dimensional) Sobolev embedding $H^2(\Omega)\hookrightarrow L^\infty(\Omega)$ and Cauchy-Schwarz inequality. If $a\leq T/2$, we have $c=0$ and we can proceed as follow. Thanks to the Br\'ezis-Gallo\"uet inequality (see Lemma \ref{lem:bg} of the Appendix), we have for almost all $s\in(a,b)$ 
\begin{align*}
\|\widetilde{u}(s)\|_{L^\infty(\Omega)} &\leq C_\Omega \|\widetilde{u}(s)\|_{H^1(\Omega)}\Big[1+\sqrt{\log(1+2\|\widetilde{u}(s)\|_{H^2(\Omega)})-\log\|\widetilde{u}(s)\|_{H^1(\Omega)}}\,\Big]\\
&\leq  C_\Omega \|\widetilde{u}(s)\|_{H^1(\Omega)}\Big[1+\sqrt{\log(1+2\|\widetilde{u}(s)\|_{H^2(\Omega)})}+\sqrt{|\log\|\widetilde{u}(s)\|_{H^1(\Omega)}|}\,\Big].
\end{align*}
Since there is a constant $C>0$ such that $z\sqrt{|\log(z)|}\leq C(z^2+\sqrt{z})$ for all $z \geq 0$, changing  the constant $C_\Omega$ if necessary eventually leads to
\begin{align*}
\|\widetilde{u}(s)\|_{L^\infty(\Omega)} &\leq C_\Omega \|\widetilde{u}(s)\|_{H^1(\Omega)}\Big[1+\sqrt{\log(1+2\|\widetilde{u}(s)\|_{H^2(\Omega)})}\,\Big] + C_\Omega(\|\widetilde{u}(s)\|_{H^1(\Omega)}^2 + \|\widetilde{u}(s)\|_{H^1(\Omega)}^{1/2}).
\end{align*}
Using Cauchy-Schwarz inequality, we thus get
\begin{align*}
\|\widetilde{u}\|_{L^1(a,b;L^\infty(\Omega))} \leq  C_{\Omega,T}(\|\widetilde{u}\|_{L^2(a,b;H^1(\Omega))})\left[1+\left(\int_a^b \log(1+2 \|\widetilde{u}(s)\|_{H^2(\Omega)}) ds\right)^{1/2}\right],
\end{align*}
for some increasing continuous function $C_{\Omega,T}$ vanishing at $0$. 
Now,  write by concavity of the logarithm
\begin{align*}
\int_a^b \log(1+2\|\widetilde{u}(s)\|_{H^2(\Omega)}) ds &=\int_a^b \log(s+2s\|\widetilde{u}(s)\|_{H^2(\Omega)}) ds -\int_a^b \log(s) ds\\
&\leq T \log\left(\frac{1}{T}\int_a^b(s+2s\|\widetilde{u}(s)\|_{H^2(\Omega)}) ds \right) -\int_a^b \log(s) ds\\
&\leq C_T\log\left (C_T + \int_0^{b}\|s\widetilde{u}(s)\|_{H^2(\Omega)}^2 ds\right) +C_T,
\end{align*}
where we used $a\leq T/2$ whence $b\leq 3T/2$. Using Theorem~\ref{thm:parabo} we (replacing $T$ by $b\leq 3T/2$ in the statement) hence get 
\begin{align*}
\int_a^b \log(1+2\|\widetilde{u}(s)\|_{H^2(\Omega)}) ds \leq D_{T,\Omega}(\|u(0)\|_{L^2(\Omega)},\|F\|_{L^2((c,b)\times\Omega)}),
\end{align*}
where $D_{T,\Omega}$ is some continuous function nondecreasing with respect to each of its arguments. Using the (yet to prove) estimate \eqref{ineq:assF}, estimate \eqref{ineq:ell} is then straightforward. 

\vspace{2mm}

To conclude, let us prove \eqref{ineq:assF}. We first have 
\begin{equation*}
\|(u_p \cdot \nabla)\widetilde{u} + (\widetilde{u}\cdot\nabla)u_p\|_{L^2(c,b;L^2(\Omega))}
\leq C_{T,\Omega}(\|u_p\|_{W^{1,\infty}}) \|\widetilde{u}\|_{L^2(c,b;H^1(\Omega))},
\end{equation*}
so that using \eqref{ineq:nrj} (with the convention $C_{\frac{3T}{2},\Omega,R} \simeq C_{T,\Omega,R}$), we get 
\begin{equation}
\label{eq-deux}
\|(u_p \cdot \nabla)\widetilde{u}+(\widetilde{u}\cdot\nabla)u_p\|_{L^2(c,b;L^2(\Omega))}
\leq D_{\Omega,T}(\|u_p\|_{W^{1,\infty}}) C_{\Omega,T,R}(\|f(c)\|_\infty,M_2f(c),\|\widetilde{u}(c)\|_2, \|\psi\|_\infty).
\end{equation}
Recall estimate \eqref{ineq:sourcel2} that we invoke here on $[c,b]$ to get 
\begin{equation*}
\|j_f - (m_0 f)u\|_{L^2(c,b;L^2(\Omega))} \leq C_\Omega(\|M_4 f\|_{L^\infty(c,b)},\|f\|_{L^\infty(c,b;L^\infty(\Omega\times\R^2))}) ( 1+ \|u\|_{L^2(c,b;L^6(\Omega))}).
\end{equation*}
Using the already proved estimates \eqref{ineq:max} and \eqref{ineq:moments} we may write 
\begin{equation} \label{ineq:sourcel2bis}
\|j_f - (m_0 f)u\|_{L^2(c,b;L^2(\Omega))}
\leq C_{\Omega,T,R,u_p}(M_4f(c),\|f(c)\|_\infty,\|\widetilde{u}(c))\|_2, \|\psi\|_\infty)  ( 1+ \|u\|_{L^2(c,b;L^6(\Omega))}).
\end{equation}
On the other hand, thanks to estimate \eqref{ineq:nrj} and the Sobolev embedding $H^1(\Omega) \hookrightarrow L^6(\Omega)$, we have 
\begin{align*}
\|u\|_{L^2(c,b;L^6(\Omega))} & \leq \|\widetilde{u}\|_{L^2(c,b;L^6(\Omega))} + D_{\Omega,T}(\|u_p\|_{W^{1,\infty}}) \\
&\leq D_{\Omega,T}(\|u_p\|_{W^{1,\infty}}) \Big[C_{\Omega,T,R}(\|f(c)\|_\infty,M_2f(c),\|\widetilde{u}(c)\|_2, \|\psi\|_\infty) + 1 \Big] \\
&\leq D_{\Omega,T}(\|u_p\|_{W^{1,\infty}}) \Big[C_{\Omega,T,R}(\|f(c)\|_\infty,M_4f(c),\|\widetilde{u}(c)\|_2, \|\psi\|_\infty) + 1 \Big],
\end{align*}
where in the last inequality we used once more $M_2f(c)\leq C \|f(c)\|_\infty + M_4f(c)$. Plugging the last estimate in \eqref{ineq:sourcel2bis} and using the already proved estimates \eqref{ineq:max} and \eqref{ineq:moments}, we 
end up with an estimate for $j_f - (m_0 f)u$ which allows to justify \eqref{ineq:assF}.
\end{proof}
The main result of this section is the following theorem.
\begin{thm}[Existence of Leray solutions] \label{thm:weak}
Let $R>0$. 
Fix $\psi \in L^{\infty}(\R_+\times \Gamma^l )$ a nonnegative entering distribution, with support in velocity included in $B(0,R)$. Consider $u_0 \in L^2(\Omega)$ with $\div \, u_0 =0$ and $f_0\in L^\infty(\Omega \times \R^2)$, a nonnegative distribution function such that $M_4 f_0 <\infty$. Then there exists a  Leray solution $(u,f)$ to the system \eqref{Kinetic}-\eqref{NavierStokes2} and boundary conditions \eqref{CBu1}-\eqref{CBf2} with  initial data $(u_0,f_0)$. 
\end{thm}
\begin{proof}[Proof of Theorem~\ref{thm:weak}]
We prove the existence of Leray solutions by an approximation procedure relying on a fixed point scheme. We fix an odd function $\chi\in\mathscr{D}(\R)$ such that $\chi(z)z\geq 0$, $|\chi(z)|\leq |z|$ and focus on the following regularized problem
\begin{gather} \label{approx:Kinetic}
\partial_{t} f + v \cdot \nabla_{x} f + \div_{v} ( \chi(u-v) f) = 0,\\
\label{approx:NavierStokes1}
\partial_{t} u + (u \cdot \nabla) u - \Delta u + \nabla p =  \int_{\R^2} f\chi(v-u) \, dv,\\
\label{approx:NavierStokes2}
\div\, u=0,
\end{gather}
where for a vector $v=(v_1,v_2)$, $\chi(v)$ means $(\chi(v_1),\chi(v_2))$.  The unknowns are $u$ and $f$, equation \eqref{approx:Kinetic} is considered on $(0,T)\times\Omega\times\R^2$ with boundary conditions \eqref{CBf1}-\eqref{CBf2}, equation \eqref{approx:NavierStokes1} is considered on  $(0,T)\times\Omega$ with boundary conditions $u=u_p$ on $\partial\Omega$. The original initial conditions are here replaced by regular approximations (still denoted $f_0$ and $u_0$). The presence of the cut-off function $\chi$ allows to write a fixed point scheme leading to the following existence result.
\begin{lem}\label{lem:approx:ex}
  Fix $T>0$. If $u_0\in H^1_0(\Omega)$ and $f_0$ is compactly supported in velocity, there exists $f\in \mathscr{C}^0([0,T]; L^\infty(\Omega \times \R^2)-w\star) \cap \mathscr{C}^0([0,T]; L^q(\Omega\times\R^2))$ for any $q<\infty$, $u\in \mathscr{C}^0([0,T];H^1(\Omega))\cap L^2(0,T;H^2(\Omega))$, $p\in L^2(0,T;H^1(\Omega))$ such that $f$ satisfies \eqref{approx:Kinetic} in the sense of Definition~\ref{def:weak} (for $t\in [0,T]$) with initial condition $f_0$, and such that $(u,p)$ solves  \eqref{approx:NavierStokes1} -- \eqref{approx:NavierStokes2} a.e. on $[0,T]\times\Omega$ and $u(0)=u_0$. This solution satisfies furthermore the estimate for any $t\in[0,T]$
\begin{multline} \label{ineq:approx:nrj}
\frac{1}{2}M_2f(t) + \frac{1}{2}\|\widetilde{u}(t)\|_2^2  +  \int_0^t \|\nabla \widetilde{u}(s)\|_2^2 \, ds
+ \frac{1}{2} \int_0^t \! \int_{\Omega\times\R^2} f \chi(u-v) \cdot (u-v) \, dv \, dx \, ds \\
\leq C_{\Omega,R,u_p}(t,\|f_0\|_\infty, M_2 f_0,\|u(0)-u_p\|_2,\|\psi\|_\infty),
\end{multline}
where $\widetilde{u}=u-u_p$ and $C_{\Omega,R,u_p}$ is a positive continuous function nondecreasing with respect to each of its arguments.
\end{lem}
\begin{proof}[Proof of Lemma~\ref{lem:approx:ex}]
Let us describe our fixed point procedure, which is quite similar to the one used in \cite{bou-gran-mou}. We start with $u\in L^2(0,T;H^{1}(\Omega))$ and define, thanks to Theorem~\ref{thm:vla} (see the Appendix), $f_u$ as the unique element of $L^\infty(0,T;L^\infty \cap L^1(\Omega\times\R^2))$ that is solution of \eqref{approx:Kinetic} together with boundary conditions \eqref{CBf1}-\eqref{CBf2} in the sense of Definition~\ref{def:weak}. Notice that since the vector field $\chi(u-v)$ is bounded and $\psi,f_0$ are compactly supported, so is $f$. One checks that the vector field 
\begin{equation*}
S(u):= \int_{\R^2} f \chi(v-u) \, dv -(u_p \cdot \nabla)u -(u\cdot\nabla)u_p ,
\end{equation*}
belongs to $L^2(0,T;L^2(\Omega))$, and we hence infer from Theorem~\ref{th:exNSpol} in the Appendix the existence of a unique solenoidal solution $\widetilde{u} \in\mathscr{C}^0([0,T];H^1_0(\Omega))\cap L^2(0,T;H^2(\Omega))\cap H^1([0,T]\times\Omega)$ such that $\widetilde{u}(0)=u_0$ and solving equation
\begin{align} \label{eq:convreg}
\partial_{t} \widetilde{u} + (\widetilde{u} \cdot \nabla) \widetilde{u} - \Delta \widetilde{u} + \nabla \widetilde{p}  = S(u).
\end{align}
The extra terms involving $u_p$ in the right-hand side $S(u)$ of the equation \eqref{eq:convreg} are added to ultimately enforce the boundary condition $u=u_p$ on $\partial\Omega$, because $\widetilde{u}$ satisfies homogeneous Dirichlet boundary conditions. Indeed, if the map $u\mapsto \widetilde{u}+u_p$ has a fixed point $u$, using  
\begin{align*}
\partial_t u_p = (u_p \cdot \nabla)u_p = -\Delta u_p + \nabla q = 0,
\end{align*}
 we see that $u$ solves \eqref{approx:NavierStokes1} with $p:=\widetilde{p}+ q$. To keep track of the dependence with respect to $u$ and $u_0$ we denote the solution $\widetilde{u}$ by
$$\Theta(u_0,S(u)):=\widetilde{u}.$$  
Recall that we have also the following estimate for $\widetilde{u}$ (see again Theorem~\ref{th:exNSpol})
\begin{multline} \label{ineq:convreg}
\|\widetilde{u}\|_{L^\infty([0,T];H^1(\Omega))}^2 + \|\widetilde{u}\|_{L^2([0,T];H^2(\Omega))}^2 
+ \|\partial_t \widetilde{u}\|_{L^2([0,T]\times \Omega)}^2 \\
\leq C_{T,\Omega}\Big(\|u_0\|_{H^1(\Omega)},\|S(u)\|_{L^2((0,T)\times\Omega))}\Big).
\end{multline}
We plan to use Schaefer's fixed point Theorem which we recall here for the reader's convenience (for a proof, see e.g. \cite{GT}):
\begin{thm}[Schaefer] \label{thm:schaefer}
Let $E$ be a Banach space and $\Lambda :E\times[0,1]\rightarrow E$ a continuous mapping sending bounded subsets of $E\times[0,1]$ on relatively compact subsets of $E$. Denoting $\Lambda_{\sigma}:=\Lambda(\cdot,\sigma)$, if $\Lambda_0=0$ and the set $\bigcup_{\sigma\in[0,1]}\textnormal{Fix}(\Lambda_\sigma)$ is bounded in $E$, then $\textnormal{Fix} (\Lambda_1) \neq \emptyset$.
\end{thm}
We consider here $E:=L^2(0,T;H^{1}(\Omega))$ and define $\Lambda$ in the following way
\begin{align*}
\Lambda : E\times\sigma &\longrightarrow E \\
 (u,\sigma) & \longmapsto  \Theta (\sigma u_0, \sigma S(u))+\sigma u_p.
\end{align*}
Here $\Lambda_0 = 0$ because of the uniqueness property of Theorem~\ref{th:exNSpol} so that we have to check the following three properties for the mapping $\Lambda$. 
%
%
\bigskip
\par
\noindent $\bullet$ \emph{$\Lambda$ sends bounded subsets on relatively compact subsets:} Starting from $(u_n)$ bounded in $E$ and $(\sigma_n) \in [0,1]^{\N}$, the corresponding sequence $(S(\sigma_n u_n))$ is bounded in $L^2(0,T;L^2(\Omega))$ (using that $(f_n)$ is uniformly bounded and compactly supported). Thanks to estimate \eqref{ineq:convreg} and the Aubin-Lions Lemma we get  that  $(\Theta(\sigma_n u_0, \sigma_n S(u_n))$ is relatively compact in $E$.
\bigskip
\par
\noindent $\bullet$  \emph{$\Lambda$ is continuous:} Assume that
$$\displaystyle(u_n,\sigma_n) \operatorname*{\longrightarrow}_{n\rightarrow +\infty} (u,\sigma) \text{ in } E\times[0,1].$$
Thanks to the previous step we know that $(\widetilde{u_n}):=(\Theta(\sigma_n u_0,\sigma_n S(u_n))$ is a relatively compact sequence in $L^2(0,T;H^1(\Omega))$ so that it just remains to prove that the only possible limit point of this sequence in this space is $\Theta(\sigma u_0,\sigma S(u))$. Consider hence $z \in E$ such a limit point. Because of estimate \eqref{ineq:convreg} we know that $(\widetilde{u_n})$ is bounded in $L^\infty(0,T;H^1(\Omega))\cap L^2(0,T;H^2(\Omega))$ and that $(\partial_t \widetilde{u_n})$ is bounded in $L^2(0,T;L^2(\Omega))$  so that, by weak (or weak$-\star$) compactness, $z$ belongs necessarily to these three spaces. On the other hand, since
$$u_n  \operatorname*{\longrightarrow}_{n\rightarrow +\infty} u \text{ in } L^1([0,T]\times\Omega),$$
(as well as for any subsequence of $(u_n)$), we get by the stability property of Theorem~\ref{thm:vla},  the strong convergence
$$f_{u_n}  \operatorname*{\longrightarrow}_{n\rightarrow +\infty} f_{u} \text{ in } L^p(0,T;L^p(\Omega\times\R^2))$$ 
for all finite values of $p$. This is sufficient to pass weakly to the limit in the equation defining $\widetilde{u_n}$, so that $z$ is eventually a solution of the equation defining $\Theta(\sigma u_0,\sigma S(u))$ and we get 
$$z=\Theta(\sigma u_0,\sigma S(u)),$$ 
by the uniqueness property stated in Theorem~\ref{th:exNSpol} of the Appendix.
\bigskip
\par
\noindent $\bullet$ \emph{$\bigcup_{\sigma\in[0,1]}\textnormal{Fix}(\Lambda_\sigma)$ is bounded:} Assume $u=\Lambda_\sigma(u)$. If $\widetilde{u}:=\Theta(\sigma u_0,\sigma S(u))$, this means that $u=\widetilde{u}+\sigma u_p$ and $\widetilde{u}$ solves the following system (we use here $(u_p \cdot \nabla)u_p = 0$)
\begin{gather*}
\partial_{t} f + v \cdot \nabla_{x} f + \div_{v} ( \chi(u -v)  f) = 0, \\
\partial_{t} \widetilde{u} + (\widetilde{u} \cdot \nabla) \widetilde{u} - \Delta \widetilde{u} + \nabla q  = \sigma \int_{\R^2} f\chi(v-u) \, dv - \sigma((u_p \cdot \nabla) \widetilde{u} +(\widetilde{u}\cdot\nabla)u_p) ,
\end{gather*}
so that, multiplying the first equation by $\sigma$,
\begin{gather*}
\partial_{t} g + v \cdot \nabla_{x} g + \div_{v} ( \chi(u -v)  g) =  0, \\
\partial_{t} \widetilde{u} + (\widetilde{u}\cdot \nabla) \widetilde{u} - \Delta \widetilde{u} + \nabla \widetilde{p}= \int_{\R^2} g\chi(v-u)) \, dv- \sigma((u_p \cdot \nabla) \widetilde{u} +(\widetilde{u}\cdot\nabla)u_p) ,
\end{gather*}
where $g:=\sigma f$. Since $\widetilde{u} \in \mathscr{C}^0([0,T];H^1_0(\Omega))$ and $\partial_t \widetilde{u} \in L^2(0,T;L^2(\Omega))$, we have enough regularity to multiply this equation  by $ \widetilde{u}$ and perform the usual integration by parts (using that both $u_p$ and $\widetilde{u}$ are solenoidal):
\begin{multline*} 
\frac{1}{2}  \|\widetilde{u}(t)\|_2^2 + \int_0^t  \|\nabla \widetilde{u}(s)\|_2^2 \, ds 
= \int_0^t \! \int_{\Omega} \int_{\R^2} g\chi(v- u)\cdot  \widetilde{u} \, dv\, dx \, ds \\
-\sigma \int_0^t \! \int_\Omega \big[(\widetilde{u} \cdot\nabla) u_p \big]\cdot  \widetilde{u}\,dx\, ds + \frac{1}{2} \|\widetilde{u}(0)\|_2^2.
\end{multline*}
Since $u_p$ is Lipschitz and $ \widetilde{u} = u- \sigma u_p$ with $\sigma\in[0,1]$, we obtain
\begin{multline} \label{eq:nrjflu:fixed}
\frac{1}{2}\|\widetilde{u}(t)\|_2^2 + \int_0^t \|\nabla \widetilde{u}(s)\|_2^2 \, ds
\leq \int_0^t \! \int_{\Omega} \int_{\R^2} g\chi(v- u)\cdot u  \, dv\, dx \, ds \\
- \sigma\int_0^t \! \int_{\Omega} \int_{\R^2} g\chi(v- u)\cdot  u_p  \, dv\, dx \, ds 
+ \|\nabla u_p\|_\infty\int_0^t \|\widetilde{u}(s)\|_2^2 \, ds + \frac{1}{2} \|\widetilde{u}(0)\|_2^2 .
\end{multline}
Now $g$ is solution of the Vlasov equation defined by the field $\chi(u-v)$ in the sense of Definition~\ref{def:weak}: Theorem~\ref{thm:vla}  applies and one gets the following estimate for the second moment of $g$
\begin{multline} \label{eq:nrjkin:fixed}
\frac{1}{2}M_2 g(t) \, ds
\leq  \int_0^t \! \int_{\Omega}\int_{\R^2} v\cdot \chi(u-v) g  \, dv \, dx \, ds  + \frac{1}{2}M_2 g_0 
- \frac{1}{2} \int_0^t \! \int_{\Gamma^l} \psi(s,x,v) |v|^2 v\cdot n(x) \,dx \, dv.
\end{multline}
Summing  \eqref{eq:nrjflu:fixed} and \eqref{eq:nrjkin:fixed} we get (using that $\chi$ is an odd function and $\text{Supp}_v \psi \subset B(0,R)$) on $[0,T]$
\begin{align*}
E(g(t),\widetilde{u}(t)) +\int_0^t \|\nabla \widetilde{u}(s)\|_2^2 ds &+ \int_0^t\int_{\Omega\times\R^2} g \chi(u-v)\cdot (u-v)  \, dv \, dx \, ds\\
&\leq\sigma \int_0^t \int_{\Omega\times\R^2} g \chi(u-v)\cdot\,u_p  \, dv \, dx \, ds \\
& +\|\nabla u_p\|_\infty\int_0^t \|\widetilde{u}(s)\|_2^2 ds  +C_{\Omega,R}(t,\|\psi\|_\infty) +  E(g(0),\widetilde{u}(0)),
\end{align*}
where $2 E(g(t),\widetilde{u}(t)) = \|\widetilde{u}(t)\|_2^2 + M_2 g(t)$.
Recall that $\chi(z)\cdot z \geq 0$ and $|\chi(z)|\leq |z|$.
Thus, using again Young's inequality together with the nonnegativity of $g$ we get
\begin{align*}
| g \chi(u-v) \cdot u_p|  &\leq \frac{1}{2} g|\chi(u-v)|^2 + \frac{1}{2}g |u_p|^2
\leq \frac{1}{2} g \chi(u-v)\cdot (u-v)   +\frac{1}{2}g |u_p|^2.
\end{align*} 
Since $\sigma\in[0,1]$, we eventually infer 
\begin{multline*}
E(g(t),\widetilde{u}(t)) + \int_0^t  \|\nabla \widetilde{u}(s)\|_2^2 \, ds 
+ \frac{1}{2} \int_0^t \! \int_{\Omega\times\R^2} g \chi(u-v)\cdot (u-v)  \, dv \, dx \, ds \\
\leq \frac{1}{2}\|u_p\|_\infty^2 M_0 g(t)
+ \|\nabla u_p\|_\infty \int_0^t  \|\widetilde{u}(s)\|_2^2 \, ds + C_{\Omega,R}(t,\|\psi\|_\infty) +  E(g(0),\widetilde{u}(0)).
\end{multline*}
We now use the moment estimate stated in Theorem~\ref{thm:vla}:
\begin{equation*}
M_0 g(t) = M_0 g(0) - \int_0^t \! \int_{\Gamma^l} \psi(s,x,v) v\cdot n(x) \,dv \,dx \,ds,
\end{equation*}
so that we deduce
\begin{multline*}
E(g(t),\widetilde{u}(t)) + \int_0^t \|\nabla \widetilde{u}(s)\|_2^2 \,ds 
 + \frac{1}{2} \int_0^t \! \int_{\Omega \times \R^2} g \chi(u-v)\cdot (u-v)  \, dv \, dx \, ds \\
 \leq C_{\Omega,R,u_p}(t,\|g(0)\|_\infty,M_2 g_0,\|\widetilde{u}(0)\|_2,\|\psi\|_\infty) 
 + \|\nabla u_p\|_\infty \int_0^t \|\widetilde{u}(s)\|_2^2 \, ds.
\end{multline*}
Thanks to Gronwall's lemma we thus get on  $[0,T]$
\begin{align} \label{ineq:nrj-fix}
\nonumber E(g(t),\widetilde{u}(t)) +\int_0^t \|\nabla \widetilde{u}(s)\|_2^2 ds &+ \frac{1}{2}\int_0^t\int_{\Omega\times\R^2} g \chi(u-v)\cdot(u-v)  \, dv \, dx \, ds\\
&\leq C_{\Omega,R,u_p}(t,\|g(0)\|_\infty,M_2 g_0,\|\widetilde{u}(0)\|_2,\|\psi\|_\infty).
\end{align}
Recall that $\sigma \in [0,1]$ and $g=\sigma f$, $\widetilde{u}(0) = \sigma(u_0-u_p)$. In particular $g(0)=\sigma f_0 \leq f_0$ and $\|\widetilde{u}(0)\|_2 \leq \|u_0-u_p\|_2$. Since $u- u_p= \widetilde{u}$, we have for $t\in[0,T]$
\begin{align} \label{ineq:nrj-fixu}
\hspace{-0.5cm}\nonumber  E(\sigma f,u- u_p)(t) +  \int_0^t \|\nabla (u- u_p)(s)\|_2^2 ds &+ \frac{1}{2}\int_0^t\int_{\Omega\times\R^2} \sigma f \chi(u-v)\cdot (u-v)  \, dv \, dx \, ds \\
&\leq C_{\Omega,R,u_p}(t,\|f_0\|_\infty, M_2 f_0,\|u(0)-u_p\|_2,\|\psi\|_\infty).
\end{align}
In particular, since $\chi(z) \cdot z\geq 0$ and due to the increasingness of $C_{\Omega,R,u_p}$ with respect to its first argument, we obtain a uniform (in $\sigma$) bound on $u$ in $E=L^2(0,T;H^1(\Omega))$. 
%
\ \par
In conclusion, we may apply Schaefer's fixed point Theorem to obtain the existence of a fixed point for the map $u\mapsto \Lambda(u,1)$. For such a fixed point, the previous estimate \eqref{ineq:nrj-fixu} is satisfied with $\sigma=1$ so that we indeed recover \eqref{ineq:approx:nrj}.
\end{proof}
\par
We aim now at considering the following asymptotics
$$ (T_\ep)_\ep \nearrow +\infty, \quad (\chi_\ep)_\ep \rightarrow \textnormal{Id}_\R, \quad  (u_0^\ep)_\ep \operatorname*{\rightarrow}^{L^2} u_0, \quad (f_0^\ep)_\ep \rightarrow f_0,$$
where the last convergence is more precisely described by $f_0^\ep = \eta^\ep f_0$, with $(\eta_\ep)$ a family of compactly supported in the velocity variable, such that $0\leq \eta^\ep \leq 1$ and increasing to the constant function $1$. For each fixed $\ep>0$, Lemma \ref{lem:approx:ex} gives us the existence of $(f_\ep,u_\ep,p_\ep)$ defined only for $(t,x,v) \in [0,T_\ep]\times\Omega\times\R^2$ and $(t,x)\in[0,T_\ep]\times\Omega$, of the following system
\begin{gather}
\label{approx:Kinetic:ep}
\partial_{t} f_\ep + v \cdot \nabla_{x} f_\ep + \div_{v} ( \chi_\ep( \eta_\ep\,u_\ep-v) f_\ep) =  0,\\
\label{approx:NavierStokes1:ep}
\partial_{t} u_\ep + (u_\ep\cdot \nabla) u_\ep - \Delta u_\ep + \nabla p_\ep = \int_{\R^2} \chi_\ep(v-\eta_\ep u_\ep) f_\ep  \, dv,\\
\label{approx:NavierStokes2:ep}
\div\, u_\ep=0, \\
(f_\varepsilon, u_\varepsilon)|_{t=0}=(f_0, u_0^\varepsilon).
\end{gather}%
These solutions satisfy furthermore estimate \eqref{ineq:approx:nrj}. We extend $(f_\ep,u_\ep,p_\ep)$ by $0$ for $t>T_\ep$. In this way $(u_\ep)$ is bounded in $L^\infty_{\textnormal{loc}}(\R_+;L^2(\Omega))\cap L^2_{\textnormal{loc}}(\R_+;H^1(\Omega))$ and $(f_\ep)$ in $\Ll^\infty(\R_+;L^\infty(\Omega\times\R^2))$. Notice however that for each $\ep>0$ the equations are only satisfied on $[0,T_\ep]$, but this is of no importance since the weak formulation of Definition~\ref{def:leray} has only to be checked on each finite interval $[0,T]$. We have weak-$\star$ compactness for $(f_\ep)_\ep$ in $\Ll^\infty(\R_+;L^\infty(\Omega\times\R^2))$ and strong compactness for $(u_\ep)$ in $L^{2}_{\textnormal{loc}}(\R_{+}; L^{2}(\Omega))$ is obtained as usual thanks to Aubin-Lions Lemma. After a diagonal extraction along the intervals $[0,n]$, we can pass to the limit in each nonlinear term to recover the weak formulation of the Leray solutions. This strong compactness together with the usual lower semi-continuity estimate for weak limits allows in the same time to get estimate \eqref{ineq:nrj} from estimate \eqref{ineq:approx:nrj} and hence check that the corresponding cluster point $(f,u)$ is indeed a Leray solution.
\end{proof}
%
%
%
%
%

%
%
%
%
%
%
%
\section{Exit geometric conditions}
\label{Sec:EGC}
\subsection{Definition of the lateral EGC}
We introduce in this section the  key geometric conditions of this paper. They  bear on the geometry of the characteristics associated to a smooth vector field $u(t,x)$ defined on $I\times \Omega$, where $I=\R_+$ or $I=\R$. \par
First, it is convenient to fix a linear extension operator $P$, continuous from $L^\infty(\Omega)$ to $L^\infty(\R^2)$ and from $\Lip(\Omega)$ to $\Lip(\R^2)$, such that
\begin{equation*}
\Supp P u \subset \mathcal{K}, \ \ \forall u \in L^{\infty}(\Omega),
\end{equation*}
where $\mathcal{K}$ is a fixed compact set of $\R^2$, containing $\Omega$. We can furthermore assume that the norm of $P$ is less than $2$. 
Also, for a regular time-dependent vector field $u(t,x)$ defined on $I \times \Omega$, we use the convention
$$
(P u) (t,\cdot) = P(u(t, \cdot)).
$$
Now, given a vector field $u(t,x)$ on $I\times \Omega$, say in $\mathscr{C}^0(I;\Lip(\Omega;\R^{2}))$, and its extension $P u$, we define the characteristics  as the solution of the following system of ordinary differential equations:
\begin{equation} \label{Characteristics}
\left\{ \begin{array}{l}
	\dot{X} = V, \\
	\dot{V} = (P u) (t,X) - V, \\
	X_{|t=s}=x, \, V_{|t=s}=v,
\end{array} \right.	
\end{equation}
with $(x,v) \in \overline{\Omega} \times \R^2$.
By the Cauchy-Lipschitz theorem, given the value $(x,v) \in \overline{\Omega} \times \R^d$ at time $s\in I$ this previous system admits a unique global solution $(X,V) \in \R^2 \times \R^2$. 
More precisely, we denote by $(X(s,t,x,v), V(s,t,x,v))$ the value of this solution at time $t$.

Now we introduce for $(x,v) \in \overline\Omega \times \R^2$ and $s \in \R_{+}$, 
\begin{align}
\label{DefTau-}
\tau_-(s,x,v) &:= \sup\{ t \in (-\infty,s)\cap I \, : \, X(s,t,x,v) \notin \overline{\Omega} \},\\
\label{DefTau+}
\tau_+(s,x,v) &:= \inf \{ t \in (s,+\infty) \,:\, X(s,t,x,v) \notin \overline{\Omega} \}.
\end{align}
The corresponding interval of times $t$ during which $X(s,t,x,v)$ remains in $\overline{\Omega}$ is therefore
$$
{\mathcal I} = [ \tau_-(s,x,v),\tau_+(s,x,v) ] .
$$
Note that this depends only on $u$ (and not on the extension operator $P$). 
Moreover, if $\tau_-(s,x,v) \neq 0$ (resp. if $\tau_+(s,x,v) < +\infty$), we have necessarily
$$
X(s,\tau_-(s,x,v),x,v) \in \partial \Omega
\quad
\text{ (resp. } X(s,\tau_+(s,x,v),x,v) \in \partial \Omega \text{)}.
$$
%
%
%
%
We are now in position to define the \emph{lateral exit geometric condition} in time $T>0$ with respect to a compact set of $K$ of $\Gamma^l$ on an time interval $J$.
%
\begin{defi}\label{Def:GC}
Let $K$ be a compact set of $\Gamma^l$ and $J$ a subinterval of $I$. 
We say that $u$ satisfies the \emph{lateral exit geometric condition} (lateral EGC) in time $T$ with respect to $K$ on $J$, if
\begin{equation} \label{def:exit1}
 \sup_{(s,x,v) \in  J \times K} (\tau_+(s,x,v)-s) <  T,
\end{equation}
and furthermore, for all $(s,x,v)\in J\times K$,  $(X,V)(s,\tau_{+}(s,x,v),x,v)\in \Sigma^+$. When $J=I$, we simply speak of the lateral EGC in time $T$ with respect to $K$.
%
\end{defi}
Loosely speaking, this definition means that all trajectories issued from a compact subset $K$ of the lateral boundary have a maximal lifetime  in $\overline{\Omega} \times \R^{2}$ that is less than $T$ and leave the domain \emph{transversally}. Note also that even if the lateral EGC is satisfied, it does not forbid some trajectories to be trapped inside $\Omega$, since it only concerns those that are issued from $K$. 
Finally we notice that this definition does not depend on the extension operator $P$. 

This geometric condition is reminiscent of the celebrated Geometric Control Condition (GCC) of Bardos, Lebeau and Rauch \cite{BLR}, which appears in the context of controllability and stabilization of the wave equation. We also mention that several GCC were recently introduced in other contexts of kinetic theory, see \cite{GHK}, \cite{BS}, \cite{HKL}. The main difference with all these geometric conditions stems from the friction term in~\eqref{Characteristics}, which has an important effect on the dynamics.


\subsection{The case of the Poiseuille flow}
\label{ex-poi}
An important particular case of vector fields we intend to consider is that of Poiseuille flows \eqref{Poiseuille}. One can indeed observe that there are compact sets $K \subset \Gamma^{l}$, such that the Poiseuille flow satisfies the lateral EGC in some time $T>0$ with respect to $K$.  Since Poiseuille flows are stationary, we may use $J=\R_+$ here. We have the following property.
\begin{lem}
\label{poigeo}
A Poiseuille flow $u_p$ satisfies the lateral EGC in some time $T>0$ with respect to a compact $K \subset \Gamma^{l}$ on $\R_+$  if $K$ satisfies the property
\begin{equation}
\label{propK}
\forall (x,v) \in K, \quad |x_2 + v_2| \neq 1.
\end{equation}
\end{lem}

\begin{proof}[Proof of Lemma~\ref{poigeo}]

Let us briefly sketch the proof of this result. Assume that \eqref{propK} is verified.
The set $K$ being compact, it means that there are $\eta_1, \eta_2>0$ such that 
\begin{equation}
\label{pppp}
\forall (x,v) \in K, \quad |x_2 + v_2| \leq 1 -\eta_1 \text{ or  Ê}Ê |x_2 + v_2| \geq 1 +\eta_2.
\end{equation}
Let us now observe  from~\eqref{Characteristics} and the fact that a Poiseuille flow has a zero vertical component, that the quantity $x_2 + v_2$ is conserved along the characteristic curves. Moreover the equation on $v_2$ can be solved explicitly. 
%
The first case in~\eqref{pppp} corresponds to the scenario where the characteristics issued from $K$ stay at a positive distance of the horizontal parts of the boundary; it follows from a view of~\eqref{Poiseuille} that the horizontal component of the Poiseuille flow along all such trajectories is bounded below in norm by a positive constant. As a consequence, all such trajectories can reach the right part of the boundary in some uniform time.
\par
The second case in~\eqref{pppp} corresponds to the scenario where the characteristics issued from $K$ can reach the horizontal parts of the boundary.

\end{proof}
\begin{rem}
\label{BC1}
With the other choice of boundary condition~\eqref{Poiautre} described in Remark~\ref{BC0}, we see that the Poiseuille flow $u'_p$ automatically satisfies the lateral EGC in some finite time with respect to \emph{any} compact set of $\Gamma^l$ on $\R_+$. Indeed, since the modulus of $u'_p$ is uniformly bounded below, all characteristics are ``uniformly'' driven to the right side.
\end{rem}
\subsection{Properties}
In this subsection, we give elementary continuity properties related to the lateral EGC.
\par
For the sake of clarity in the coming lines we will use the notation: for $(x,v) \in \overline{\Omega} \times \R^2$, 
\begin{equation}
\label{eq:not:tauxv}
\tau_{x,v}^\pm:=\tau_{\pm}(0,x,v),
\end{equation}
and also (when the latter is finite)
\begin{equation}\label{eq:not:XVxv}
(X_{x,v}^\pm,V_{x,v}^\pm):=(X(0,\tau_{x,v}^\pm,x,v),V(0,\tau_{x,v}^\pm,x,v)).
\end{equation}
We start with a lemma concerning the regularity of the entering and exit times of a particle inside $\Omega$. We recall that $\Sigma^s$ was defined in \eqref{Eq:DefSigmas}. The following result concerns only stationary vector fields (and we thus consider the case $I=\R$)
\begin{lem}\label{propo:tau}
Consider $u \in \Lip(\overline{\Omega})$ and $K$ a compact subset of $\Gamma^l$ 
for which we define 
\begin{equation}\label{AK}
A_K:=\{ (x,v) \in \overline{\Omega} \times \R^{2} \,:\, \tau_{x,v}^- \neq -\infty \text{ and } (X_{x,v}^-,V_{x,v}^-) \in K\}.
\end{equation}
We have the following properties
\begin{itemize}
\item[(i)] The maps $(x,v)\mapsto \tau_{x,v}^-$ and $(x,v) \mapsto (X_{x,v}^-,V_{x,v}^-)$ are continuous at any point $(x,v)$ such that $\tau_{x,v}^-\neq -\infty$ and $(X_{x,v}^-,V_{x,v}^-) \notin \Sigma^s$.
\item[(ii)] The maps $(x,v)\mapsto \tau_{x,v}^+$ and $(x,v) \mapsto (X_{x,v}^+,V_{x,v}^+)$ are continuous at any point $(x,v)$ such that $\tau_{x,v}^+\neq +\infty$ and $(X_{x,v}^+,V_{x,v}^+) \notin \Sigma^s$.
\item[(iii)] If  $u$ satisfies the lateral EGC with respect to $K$ in a finite time, then $A_K$ is at positive distance of the set $\{ (x,v) \in \overline{\Omega} \times \R^{2} \,:\, \tau_{x,v}^- = -\infty \}$.
\end{itemize}
\end{lem}
\begin{proof}[Proof of Lemma~\ref{propo:tau}] 
\ \par
\begin{itemize}
\item[$(i)$ and $(ii)$] The proof are the same so we focus on $(i)$. Using Gronwall's lemma, we get the following inequality for any pairs $(x,v)$ and $(z,w)$ of $\overline{\Omega}\times\R^2$ and negative times $s$:
\begin{align} \label{ineq:gron}
\sup_{[s,0]} |(X,V)(0,\cdot,x,v)-(X,V)(0,\cdot,z,w)|\leq C(|s|) |(x,v)-(z,w)|,
\end{align}
where $C$ is some nondecreasing positive function. By definition of $\tau_{x,v}^-$ and since $(X_{x,v}^-,V_{x,v}^-)\notin\Sigma^s$, we have that  $X(0,s,x,v)$ leaves $\overline{\Omega}$ for $s< \tau_{x,v}^-$ , we have also $X(0,s,x,v)\in \Omega$ for $s>\tau_{x,v}^-$. Consequently, for any $\ep>0$, if $|(x,v)-(z,w)|$ is small enough, using \eqref{ineq:gron} we get first $\tau_{z,w}^- \neq -\infty$ and more precisely that $X(0,s,z,w)$ remains in $\Omega$ for $s > \tau_{x,v}^-+\ep$, and leaves $\overline{\Omega}$ for $s < \tau_{x,v}^--\ep$ (because $V(0,\cdot,z,w) \cdot n(x) <0 $ around $\tau_{x,v}^-$). This implies that $|\tau_{z,w}^--\tau_{x,v}^-| < 2 \ep$.
In particular, using \eqref{ineq:gron} on the interval $[\tau_{x,v}^--2\ep,\tau_{x,v}^-+2\ep]$ we get $|(X_{x,v},V_{x,v})-(X_{z,w},V_{z,w})|$ arbitrarily small.
\item[$(iii)$] We first notice three useful facts:
\begin{itemize}
\item[$\bullet$] If $(x,v)\in A_K$, thanks to the lateral EGC (say, in time $T$), we know that $|\tau_{x,v}^-| \leq T$.
\item[$\bullet$] Due to Equation~\eqref{Characteristics}, any characteristic curve $(X,V)$ satisfies that for any $s,t \in \R$, $|e^t V_1(t)- e^s V_1(s)| \leq \|Pu\|_\infty |t-s|$.
\item[$\bullet$] Since $K$ is a compact subset of $\Gamma^l$, we have $K\subset \{v_1>\alpha\}$ for some positive number $\alpha$. 
\end{itemize}
We infer from these observations the existence of $\delta>0$ such that, for any $(x,v)\in A_K$, if $|s-\tau_{x,v}^-|<\delta$ then $V_1(0,s,x,v)\geq \alpha/2$. Therefore, $X_1(0,\tau_{x,v}^--\delta,x,v) \leq -L - \frac{\alpha}{2}\delta$. Thanks to inequality \eqref{ineq:gron} (and because $C(|\tau_{x,v}^--\delta|)\leq C(T+|\delta|)$) we get the existence of $\eta >0$ such that for any $(x,v)\in A_K$ and $(z,w)\in\overline{\Omega}\times\R^2$, $|(x,v)-(z,w)|< \eta $ implies $|X_1(0,\tau_{x,v}^--\delta,x,v)-X_1(0,\tau_{x,v}^-\delta,z,w)| < \alpha\delta/4$ from which we obtain $\tau_{z,w}^- \neq -\infty$.
\end{itemize}
\end{proof}
\begin{rem}
Point (iii) is in general false without an assumption such as the lateral EGC. It avoids indeed the scenario in which an exiting characteristic curve asymptotically loops around a closed one.
\end{rem}

We continue with the following result, which expresses the fact that a small perturbation of a stationary vector field $u^\sharp$ satisfying the lateral EGC  satisfies the same condition, in a possibly longer time. 
\begin{lem} \label{Lem:Gronwall}
Fix $T>1$. Consider $u^\sharp$ a stationary vector-field satisfying the lateral EGC in time $T-1$  with respect to $K\subset  \Gamma^{l}$ on $\R^+$. Let $J$ be an interval of $\R_+$. There is $\delta>0$ such that any $u\in L^\infty(\R_+;\Lip(\Omega))$ such that
$$
\forall t \in J,\qquad \int_{t}^{t+T} \| u(\tau,\cdot) - u^\sharp\|_\infty d \tau  \leq \delta,
$$
satisfies the lateral EGC in time $T$ with respect to $K$  on the interval $J$.
\end{lem}
\begin{proof}[Proof of Lemma~\ref{Lem:Gronwall}] 
Consider $u^{\sharp}$ and $u\in L^\infty(\R_+;\Lip(\Omega))$ as above. For all $(x,v) \in K$ and $(s,t)\in\R_+\times\R_+$, we can introduce
\begin{itemize} 
\item $(X(s,t,x,v), V(s,t,x,v))$  the characteristics associated to $u$ and 
\item $(X^\sharp(s,t,x,v), V^\sharp(s,t,x,v))$  the characteristics associated to $u^\sharp$.
\end{itemize}
Define first $\tau_+^\sharp$ the exit time associated to $u^\sharp$. 
By the lateral EGC satisfied by $u^\sharp$, the regularity of the map $(x,v) \mapsto (X_{x,v}^{\sharp +},V_{x,v}^{\sharp +},\tau_{x,v}^{\sharp +})$ (that is point (ii) of Lemma~\ref{propo:tau}) and the fact that all characteristics associated to $u^\sharp$ and issued from $K$ exit transversally $\Omega \times \R^2$ in a time less or equal to $T-1$, we find by a compactness argument the existence of $\eta,\ep,\kappa>0$ such that for all $(x,v) \in K$ and all $s \in \R$,  such that
 \begin{align}
 \label{0}
X^\sharp(s,\tau_{x,v}^{\sharp +}-\ep,x,v)&\notin (\R^2 \setminus\Omega) + B(0,\eta), \\
 \label{1}
X^\sharp(s,\tau_{x,v}^{\sharp +}+\ep,x,v)&\notin \Omega + B(0,\eta),\\
\label{2}|t-\tau^{\sharp +}_{x,v}|< \ep &\Rightarrow V^{\sharp}(s,t,x,v) \cdot n(X^{\sharp +}_{x,v}) \geq \kappa >0.
 \end{align}

%
We fix  $(s,x,v) \in J \times K$. 
We set $(Y,W)= (X(s,t,x,v)- X^\sharp(s,t,x,v), V(s,t,x,v)-V^\sharp(s,t,x,v))$. We observe that $(Y,W)$ solves
\begin{equation} \label{EDO}
\left\{ \begin{array}{l}
	\dot{Y} = W, \\
	\dot{W} = (Pu(t,X)- Pu^\sharp(t,X^\sharp)) - W, \\
	Y_{|t=s}=0, \, W_{|t=s}=0.
\end{array} \right.	
\end{equation}
Recall that the norm of the extension operator $P$ is less than $2$. We hence have for all times $t\geq s$
\begin{equation*}
 |(Y,W) | \leq |Y|+|W| \leq 2\int_{s}^t |W(\tau)| d\tau + 2\int_{s}^t \|u(\tau,\cdot)-u^\sharp\|_\infty d\tau + 2 \|\nabla u^\sharp\|_\infty \int_{s}^t |Y(\tau)| d\tau,
\end{equation*}
in which $|\cdot|$ denotes the Euclidean norm in $\R^2$ or $\R^4$. In particular $|Y|+|W|\leq 2 |(Y,W)|$ and by  Gronwall's lemma, we infer that
\begin{equation*}
|(Y,W)| (t)  \leq 2\left[\exp ( 4 (1+ \| \nabla u^\sharp \|_\infty)|t-s|) \right]  \int_{s}^t \| u(\tau,\cdot) - u^\sharp\|_\infty d \tau.
\end{equation*}
We set 
$$
\delta := \frac{\eta/4}{\exp ( 4 (1+ \| \nabla u^\sharp \|_\infty)T)}.
$$
Therefore as soon as $u$ satisfies  
$\forall s \in J,\displaystyle \int_{s}^{s+T} \| u - u^\sharp\|_\infty d \tau  \leq \delta$,
we have in particular
$$
\sup_{t \in [s ,s+ T ]} |(Y,W)| (t) \leq \eta/2. 
$$
Recall that by~\eqref{0}, there is $t_0 \in (s,s + T ]$ such that $(X^\sharp,V^\sharp)(s,t_0,x,v ) \notin \overline{\Omega} \times \R^2 + B(0,\eta)$. 
We deduce that $(X,V)(s,t_0,x,v ) \notin \overline{\Omega} \times \R^2 $, and with \eqref{0} -- \eqref{1} -- \eqref{2}  one can also check that the exit is done transversally. 
Consequently, the EGC is satisfied for $u$ in time $T$ with respect to $K$ on the interval $J$.

\end{proof}
\begin{rem}
\label{remT}
Keeping the same notations as in the statement of the Lemma, it is clear from a view of the proof that the result still holds if we replace in the conclusion $T$ by $T-1 +\alpha$ for any $\alpha>0$.
\end{rem}
%
%
The next proposition considers the dependence of the objects defined above with respect to a change in the vector field $u$.
\begin{lem} \label{propo:tau:n}
Consider $(u_n)$ a sequence of $\Lip({\Omega})$ uniformly converging to an element $u \in \Lip({\Omega})$. If $K$ is a compact set of $\Gamma^{l}$ with $A_K$ defined as in~\eqref{AK} we have the following properties (with obvious $n$-variant notations) 
\begin{itemize}
\item[(i)] If $\tau_{x,v}\neq -\infty$ and $(X_{x,v},V_{x,v})\notin\Sigma^s$, $(\tau^n_{x,v},X^n_{x,v},V^n_{x,v})$ converges to $(\tau_{x,v},X_{x,v},V_{x,v})$.
\item[(ii)] If  $u$ satisfies the lateral EGC with respect to $K$ in a finite time and $\tau_{x,v}=-\infty$, then for $n$ large enough, $(x,v)\notin A_K^n$.
\end{itemize}
\end{lem}
\begin{proof}[Proof of Lemma~\ref{propo:tau:n}]

\begin{itemize}
\item[$(i)$]
Since $(u_n)$ converges to $u$ uniformly,  Gronwall's lemma (relying on the Lipschitz constant of $u$) implies that, for any $(x,v) \in \overline{\Omega} \times \R^2$, the characteristics $t \mapsto (X^n,V^n)(0,t,x,v)$ converge uniformly on compact sets to $t\mapsto (X,V)(0,t,x,v)$.
 Since $(X_{x,v},V_{x,v})\notin \Sigma^s$, for $\ep>0$ small enough, the compact set $\{(X,V)(0,t,x,v)\,:\,t\in[\tau_{x,v}+\ep,0]\}$ lies in $\Omega \times \R^2$. Thanks to the previous convergence, one infers (for $n$ large enough) that $\{(X^n,V^n)(0,t,x,v) \,:\, t \in [\tau_{x,v}+\ep,0] \} \subset \Omega \times \R^2$, thus $\tau_{x,v}^n \leq \tau_{x,v} + \ep$. On the other hand, the very definition of $\tau_{x,v}$ implies (replacing $\ep$ by a smaller quantity if necessary) $(X,V)(0,\tau_{x,v}-\ep,x,v) \notin \overline{\Omega} \times \R^2$: as before, for $n$ large enough, $(X^n,V^n)(0,\tau_{x,v}-\ep,x,v) \notin \overline{\Omega} \times \R^2$, that is $\tau_{x,v}^n\geq \tau_{x,v} -\ep$.
All in all, we obtained the convergence $\tau_{x,v}^n \rightarrow \tau_{x,v}$ as $n \rightarrow +\infty$.
Since $t\mapsto(X^n,V^n)(0,t,x,v)$ converges uniformly on compact sets to $t \mapsto (X,V)(0,t,x,v)$, one has also that $(X^n_{x,v},V^n_{x,v})=(X^n,V^n)(0,\tau^n_{x,v},x,v)$ converges to $(X_{x,v},V_{x,v})=(X,V)(0,\tau_{x,v},x,v)$.
\item[$(ii)$] Assume $\tau_{x,v}=-\infty$ and  that $(x,v) \in A_K^{n_k}$ for some subsequence indexed $(n_k)_{k \in \N}$.
If $u$ satisfies the lateral EGC at time $T$ with respect to $K$, thanks to Lemma~\ref{Lem:Gronwall}, we have the existence of $\alpha>0$ such that $u_{n_k}$ also satisfies the lateral EGC at time $T+1$ (for $k$ large enough).
It follows that the sequence $(\tau_{x,v}^{n_k})_{k \in \N}$ is bounded and thus converges to some $\tau$, up to an unlabeled extraction. Using as before the local uniform convergence of the characteristics curves, we infer the convergence of $(X^{n_k},V^{n_k})(0,\tau_{x,v}^{n_k},x,v)$ towards $(X,V)(0,\tau,x,v)$ as $k \rightarrow +\infty$. Since $(x,v)\in A_K^{n_k}$ we have  $(X^{n_k},V^{n_k})(0,\tau_{x,v}^{n_k},x,v)\in K$, and since $K$ is closed this entails $(X,V)(0,\tau,x,v)\in K$, and in particular $X(0,\tau-\ep,x,v)\notin \overline{\Omega}$ for $\ep$ small enough, contradicting $\tau_{x,v}=-\infty$.
\end{itemize}
\end{proof}
\subsection{Variants of the exit geometric condition}

In view of the final stability analysis, we shall also need some variants of the lateral EGC we have just introduced, with the aim to handle compact sets $K$ of  $\overline{\Omega} \times \R^2$. 

The first one is a straightforward generalization of Definition~\ref{Def:GC}:

\begin{defi}\label{Def:GC'}
Let $K$ be a compact set of $\overline\Omega \times \R^2$ and $J$ a subinterval of $\R_{+}$. 
We say that $u$ satisfies the \emph{internal lateral exit geometric condition} (internal lateral EGC) in time $T$ with respect to $K$ on $J$, if
\begin{equation} \label{def:exit1'}
 \sup_{(s,x,v) \in  J \times K} (\tau_+(s,x,v)-s) <  T,
\end{equation}
and furthermore, for all $(s,x,v)\in J\times K$,  $(X,V)(s,\tau_{+}(s,x,v),x,v)\in \Sigma^+$.
%
\end{defi}
As for the lateral EGC of Definition~\ref{Def:GC}, if $u$ is a Poiseuille flow, Lemma~\ref{poigeo} is still relevant: for any compact set $K$ satisfying the condition~\eqref{propK}, $u$ satisfies the internal lateral EGC in some time $T>0$ with respect to $K$ on $\R^+$. We also have the exact analogue of Lemma~\ref{Lem:Gronwall} (with almost the same proof).
\begin{lem} \label{Lem:Gronwall'}
Fix $T>1$. Consider $u^\sharp$ a stationary vector-field satisfying the internal lateral EGC in time $T-1$  with respect to $K\subset  \overline\Omega \times \R^2$ on $\R^+$
and $J$ an interval of $\R_+$. There is $\delta>0$ such that any $u\in L^\infty(\R_+;\Lip(\Omega))$ such that
$$
\forall t \in J,\qquad \int_{t}^{t+T} \| u(\tau,\cdot) - u^\sharp\|_\infty d \tau  \leq \delta,
$$
satisfies the internal lateral EGC in time $T$ with respect to $K$  on the interval $J$.
\end{lem}
\par

The other variant of the EGC focuses only on characteristics starting at time $0$, and is thus called the \emph{initial EGC}
\begin{defi}\label{Def:GC2}
Let $K$ be a compact set of $\overline{\Omega} \times \R^2$. We say that $u$ satisfies the \emph{initial exit geometric condition} (initial EGC) in time $T$ with respect to $K$, if
\begin{equation} \label{def:exit2}
 \sup_{(x,v) \in K} \tau_+(0,x,v) <  T,
\end{equation}
and furthermore, for all $(x,v)\in  K$,  $(X,V)(0,\tau_{+}(0,x,v),x,v)\in \Sigma^+$.
%
\end{defi}
Again if $u$ is a Poiseuille flow, an adaptation of Lemma~\ref{poigeo} is available for the initial EGC. We also have the analogue of Lemma~\ref{Lem:Gronwall} (and the proof is the same as well):
\begin{lem} \label{Lem:Gronwall2}
Fix $T>1$. Consider $u^\sharp$ a stationary vector-field satisfying the initial EGC with respect to $K\subset \overline{\Omega} \times \R^2 $ 
in time $T-1$. There is $\delta>0$ such that any $u\in L^\infty(\R_+;\Lip(\Omega))$ such that
$$
 \int_{0}^{T} \| u(\tau,\cdot) - u^\sharp\|_\infty d \tau  \leq \delta,
$$
satisfies the initial EGC in time $T$ with respect to $K$.
\end{lem}
Let us mention to conclude that Remark~\ref{remT} is still relevant for Lemmas~\ref{Lem:Gronwall'} and \ref{Lem:Gronwall2}. The analogues of Lemmas~\ref{propo:tau} and~\ref{propo:tau:n} hold as well, but we shall not study them, as they will not be needed  in the following.
%
%
%
%
%
%
%
\section{Existence of nontrivial equilibria for the Vlasov-Navier-Stokes system in the pipe}
\label{Sec:Equilibria}
In this section, we establish the existence of regular stationary states for the Vlasov-Navier-Stokes system 
such that $f \neq 0$. In this section the only useful EGC is the lateral one (see Definition \ref{def:exit1}).
\subsection{Statement of the result}
Let us first introduce some appropriate function spaces. We define a weighted $L^\infty$ space as follows:
$$
\mathscr{L}_{\gamma^{-1}}^{\infty}(\Omega) :=\left\{ u \in L^\infty(\Omega), \, \gamma^{-1}{u} \in L^\infty(\Omega) \right\}, 
$$
where
\begin{equation*}
\gamma(x_1,x_2):=1-x_2^2.
\end{equation*}
We also introduce ${\mathcal E}$ and $\overline{\mathcal E}$ as the following vector spaces:
%
$$
{\mathcal E} := \mathscr{C}^1(\overline{\Omega}) \cap \mathscr{L}_{\gamma^{-1}}^{\infty}(\Omega)
\ \text{ and } \ 
\overline{\mathcal E} := \Lip({\Omega}) \cap \mathscr{L}_{\gamma^{-1}}^{\infty}(\Omega).
$$
We endow both spaces with the following norm: for $u \in \overline{\mathcal{E}}$,
\begin{equation*}
\|u\|_{{\mathcal E}}:=\|u\|_\infty + \|\nabla u\|_\infty + \| \gamma^{-1} u\|_\infty.
\end{equation*}
Note that, due to the weight $\gamma^{-1}$ the functions of $\mathcal{E}$ and $\overline{\mathcal{E}}$ vanish at the boundary $x_2\in\{-1,1\}$. Observe in particular that Poiseuille flows belong to $\mathcal{E}$. 
  \par
\ \par
The main result of this section is the following theorem, establishing the existence of nontrivial stationary solutions to the Vlasov-Navier-Stokes system in the pipe. 
The framework we consider here is more general than the one presented in the introduction, as we allow other boundary conditions than~\eqref{CBu1}-\eqref{CBu2}.
\begin{thm} \label{Theo:Equilibria}
Let $T>1$, $R>0$ and $\varepsilon >0$. There exist some constants $C_1(\Omega), C_2(\Omega,T,\varepsilon,R)>0$ such that the following holds.
Let $u^{\sharp} \in {\mathcal E}$ be a stationary solution of the Navier-Stokes equation \eqref{SNS1}-\eqref{SNS2} and $\psi \in \mathscr{C}^{0}_{c}\cap \Lip(\Gamma^l)$ an incoming distribution function such that 
\begin{equation} \label{Condition}
\| u^{\sharp} \|_{W^{1,\infty}} \leq C_1(\Omega )\ \text{ and } \  \| \psi \|_{L^{\infty}} \leq C_2(\Omega,T,\varepsilon,R),
\end{equation}
and such that the support of $\psi$ is  included in the ball $\{|v|\leq R\}$.
Assume finally that $u^{\sharp}$ satisfies the lateral EGC in time $T-1$ with respect to the support of $\psi$ on $\R$.
Then there exists a stationary solution $(\overline{u},\overline{f})$ in ${\mathcal E} \times L^{\infty}(\Omega \times \R^{2})$ of \eqref{Kinetic}-\eqref{NavierStokes2} with the following boundary conditions
\begin{gather}
\label{CBStat1}
\overline{u} = u^{\sharp} \ \text{ on } \partial\Omega, \\
\label{CBStat2}
\overline{f} = \psi \ \text{ on } \Gamma^l, \\
\label{CBStat3}
\overline{f} = 0 \ \text{ on } \Gamma^r \cup \Gamma^u \cup \Gamma^d,
\end{gather}
such that $\overline{f}$ is compactly supported in $\overline{\Omega} \times \R^2$ and
\begin{equation} \label{Petitesse}
\| \overline{u} - u^{\sharp} \|_{{\mathcal E}} + \| \overline{f} \|_{L^{\infty}(\Omega \times \R^{2})} \leq \varepsilon.
\end{equation}
Moreover,
%
%
then so is $\overline{f}$ with the following estimate  
\begin{equation} \label{fbarlip}
\|\nabla_{x,v} \overline{f}\|_\infty \leq C(T,  u^\sharp, \psi) \| \psi \|_{W^{1,\infty}}.
\end{equation}
\end{thm}
In particular, in this result, we can take the Poiseuille flow (for $u_{\rm max}$ small enough) as the vector field $u^\sharp$, together with a suitable compact support for $\psi$ (see the discussion in Section~\ref{ex-poi}). 
%
%
The rest of Section~\ref{Sec:Equilibria} is dedicated to the proof of Theorem~\ref{Theo:Equilibria}.
\subsection{Constants}\label{subsec:const}
Let us specify a bit more the constants appearing in \eqref{Condition}. The elliptic regularity of the Stokes operator on $\Omega$ (see \cite{frosties} for elliptic regularity estimates in convex polygons), together with the Sobolev embedding $W^{2,3}(\Omega)\hookrightarrow \mathscr{C}^1(\overline{\Omega})$ gives us the existence of $C_{\textnormal{St},\Omega}>0$ such that 
\begin{equation}\label{ineq:stokes}
\|w\|_{W^{1,\infty}(\Omega)} \leq C_{\textnormal{St},\Omega} \|F\|_\infty,
\end{equation}
where $w$ is the solenoidal solution of the Stokes equation
\begin{align*}
-\Delta w + \nabla q &= F,\\
 w_{|\partial\Omega} &= 0,
\end{align*}
with a bounded right hand side.The first constant appearing in \eqref{Condition} is 
\begin{equation}\label{eq:c1}
C_1(\Omega):=\frac{1}{12C_{\textnormal{St},\Omega}}.
\end{equation}

\vspace{2mm} 

Next, applying Lemma~\ref{Lem:Gronwall} to the vector-field $u^\sharp$ with $J=\R$, we get the existence of $\delta>0$ such that any element of $\overline{B}_{\overline{\mathcal{E}}}(u^\sharp,\frac{\delta}{T})$ (thus not depending on time) satisfies the lateral EGC in time $T$ with respect to to $\textnormal{Supp}~\psi$. Without loss of generality we can assume in the statement of Theorem~\ref{Theo:Equilibria} that
\begin{equation}\label{ineq:epetit}
\ep < \min\left(1,\frac{\delta}{T},\frac{1}{6 C_{\textnormal{St},\Omega}}\right).
\end{equation}
 We will consider the ball $\overline{B}_{\overline{\mathcal{E}}}(u^\sharp,\ep)$ in the next subsection to establish a fixed point procedure that will ultimately lead to Theorem~\ref{Theo:Equilibria}.

\vspace{2mm} 

Lastly, for $C_{2}(T,\Omega,R,\ep)$, the expression is a bit more intricate. If $M:=R + T(1+C_1(\Omega))$, we take 
\begin{equation}
\label{eq:c2} C_2(T,\Omega,R,\ep) = \frac{1}{6C_{\textnormal{St},\Omega}}e^{-2T} \min(\ep,2\pi M^3).
\end{equation}
\subsection{Fixed point operator}
Theorem~\ref{Theo:Equilibria} is proved through a fixed point scheme. In this subsection, we introduce the corresponding fixed point operator.
The operator is denoted by $\Lambda$, and is defined on the closed ball $\overline{B}_{\overline{\mathcal E}}(u^{\sharp},\ep)$ of $\overline{{\mathcal E}}$. Here and in what follows, $u^\sharp$ systematically stands for a vector field satisfying the assumptions of Theorem~\ref{Theo:Equilibria} (together with $\varepsilon$). 
%
%

%
%
\bigskip
Now the fixed point operator $\Lambda$ is defined as follows. \par
\ \par
\noindent
\textbf{First part.}
Given $u \in \overline{B}_{\overline{\mathcal E}}(u^{\sharp},\ep)$, we first associate a stationary distribution function $g \in L^{\infty}(\Omega \times \R^{2})$ defined as follows. Let $(X,V)$ be the characteristics associated to $u$ through the system \eqref{Characteristics}.

 If $K:=\text{Supp}\, \psi$, recall the notations $\tau^-_{x,v}$ and $(X^-_{x,v},V^-_{x,v})$ introduced in formulae \eqref{eq:not:tauxv} -- \eqref{eq:not:XVxv} and the following subset of $\overline{\Omega}\times\R^2$
\begin{equation} \label{DefAK}
A_K:=\{\tau^-_{x,v}\neq -\infty\text{ and }(X^-_{x,v},V^-_{x,v}) \in K\}.
\end{equation}
For the sake of readability, until the end of the section, we shall drop the $-$, writing $\tau_{x,v}$ and $(X_{x,v},V_{x,v})$ for $\tau^-_{x,v}$ and $(X^-_{x,v},V^-_{x,v})$.
\ \par
We define $g$ on $\overline{\Omega}\times\R^2$ by the formula
\begin{align} \label{PropagationCaracteristiques}
g(x,v)= \exp(-2\tau_{x,v}) \, \psi(X_{x,v},V_{x,v}) \mathbf{1}_{A_K}(x,v).
\end{align}
If $(x_t,v_t):=(X(0,t,x,v),V(0,t,x,v))$, on the one hand one has $\tau_{{x_t},{v_t}}=\tau_{x,v}-t$ and on the other hand $X_{x_t,v_t}$ and $V_{x_t,v_t}$ do not depend on $t$ since they correspond to the entering point in the phase space of the characteristic curve $t\mapsto (x_t,v_t)$. Likewise, the value of $\mathbf{1}_{A_K}(x_t,v_t)$ does not depend on $t$ so that $t\mapsto e^{-2t} g(x_t,v_t)$ is constant. Conversely, one checks that $g$ is the only function defined on $\overline{\Omega}\times\R^2$ such that $t\mapsto e^{-2t}g(x_t,v_t)$ is constant, $g=\psi\mathbf{1}_{\Gamma^l}$ on $\Sigma^{-}$ and $g\mathbf{1}_{\{ \tau_{x,v}=-\infty \}}=0$.
In particular, if one manages to prove that $g$ is Lipschitz, then it solves automatically the following system a.e.
\begin{gather}
\label{StatKinetic}
v \cdot \nabla_x g + \div_{v} ((u-v) g) = 0 \text{ for } (x,v) \text{ in } \Omega \times \R^{2}, \\
\label{StatCB}
g(x,v)= \psi(x,v) \text{ for } (x,v) \in \Gamma^l \ \text{ and } \ g(x,v)= 0 \text{ for } (x,v)\in \Gamma^r \cup \Gamma^u \cup \Gamma^d.
\end{gather}
We will not use the fact that $g$ is a solution of the previous equation to prove the existence of a fixed point, but only the definition of $g$ given by \eqref{PropagationCaracteristiques}, that is why we postpone the proof of Lipschitz regularity of $g$ to Subsection~\ref{subsec:lipreg}. \par
%
%
%
%
%
\ \par
\noindent
\textbf{Second part.} In a second time, to the distribution function $g$, we associate $\hat{u}$ (aimed at belonging to $\overline{B}_{\overline{\mathcal E}}(u^{\sharp},\varepsilon)$) as the solution of the following Stokes system:
\begin{gather}
\label{NSL1}
- \Delta \hat{u} + \nabla \hat{p} = -(u \cdot \nabla) u + \int_{\R^{2}} g(v-u) \, dv \ \text{ in } \Omega, \\
\label{NSL2}
\div\, \hat{u} =0  \ \text{ in } \Omega,
\end{gather}
with boundary conditions
\begin{equation} \label{NSLBC}
\hat{u} = u^{\sharp} \ \text{ on } \partial \Omega.
\end{equation}
{\bf Conclusion.} We finally set $\Lambda u := \hat{u}$.
\subsection{Existence of a fixed point}
To prove the existence of a fixed point to $\Lambda$, we will use Schauder's fixed point theorem. First, we endow $\overline{B}_{\overline{\mathcal E}}(u^{\sharp},\ep)$ with the $L^{\infty}$ topology. Hence the fact that the convex set $\overline{B}_{\overline{\mathcal E}}(u^{\sharp},\ep)$ is compact is a straightforward consequence of the Arzel\`a-Ascoli theorem. 
Now to prove that $\Lambda$ has a fixed point, it hence suffices to prove the two following statements.
\begin{lem} \label{LemLambdaBall}
The operator $\Lambda$ sends $\overline{B}_{\overline{\mathcal E}}(u^{\sharp},\ep)$ into itself (which includes the fact that $\Lambda$ is well defined).
\end{lem}
\begin{lem} \label{LemLambdaCont}
The operator $\Lambda$ is continuous.
\end{lem}
The rest of this subsection is dedicated to the proof of these two lemmas.
\begin{proof}[Proof of Lemma~\ref{LemLambdaBall}]
Let $u \in \overline{B}_{\overline{\mathcal E}}(u^{\sharp},\ep)$. We start by estimating the $L^\infty$ norm of $g$ as well as its support in velocity.
To this end, we use the very definition of $g$ in \eqref{PropagationCaracteristiques}. Using Lemma~\ref{Lem:Gronwall}, we know that (that was the purpose of \eqref{ineq:epetit}) for any $(x,v) \in \Omega \times \R^2$, $|\tau_-(0,x,v)|<T$. This implies that 
$$\|g \|_{\infty} \leq \exp(2T)\|\psi\|_{\infty} \leq \ep,$$ 
where we have used \eqref{Condition} with the explicit expression \eqref{eq:c2} for the second inequality. For what concerns the support, formula~\eqref{PropagationCaracteristiques} also implies that $\Supp_v g \subset \{|v| \leq R + T\| u\|_\infty\}$. Since $\|u\|_\infty \leq \ep + \|u^\sharp\|_\infty$, using $\ep<1$ and the assumption $\|u^\sharp\|_\infty \leq C_1(\Omega)$, we eventually have 
$$\Supp_v g \subset B(0,M),$$ 
where $M:=R +T(1+C_1(\Omega))$. 
We use again the notation $\hat{u}=\Lambda u$.
Observe now that $\hat{u}- u^\sharp$ satisfies the equation
$$
-\Delta (\hat{u}- u^\sharp) + \nabla (\hat{p}- p^\sharp) =  -(u\cdot \nabla) u + (u^\sharp\cdot \nabla) u^\sharp +  \int_{\R^{2}} g(v-u) \, dv.
$$
We then have the estimate
\begin{align*}
\left\|-(u \cdot \nabla) u + (u^{\sharp} \cdot \nabla) u^{\sharp} \right\|_{\infty} &= \left\|((u^\sharp-u) \cdot \nabla ) u^\sharp + (u \cdot \nabla ) (u^\sharp-u) \right\|_{\infty}\\
&  \leq  \ep(\ep +2 \|u^{\sharp}\|_{W^{1,\infty}(\Omega)}), 
\end{align*}
as well as
\begin{align*}
\left\| \int_{\R^{2}} g(v-u) \, dv\right\|_\infty &\leq \exp(2T)\|\psi\|_{\infty}  \left( \pi M^3 + \pi  M^2\|u\|_\infty\right)\\
&\leq 2 \pi M^3 \exp(2T)\|\psi\|_\infty.
\end{align*}
Using the assumption \eqref{Condition} with the explicit expressions \eqref{eq:c1} and \eqref{eq:c2}, the previous inequalities lead to
\begin{align*}
\left\|-(u \cdot \nabla) u + (u^{\sharp} \cdot \nabla) u^{\sharp} + \int_{\R^{2}} g(v-u) \, dv\right\|_\infty \leq  \frac{1}{2} \frac{\ep}{C_{\textnormal{St},\Omega}}.
\end{align*}
Recalling the elliptic estimate \eqref{ineq:stokes} involving the constant $C_{\textnormal{St},\Omega}$, we infer 
\begin{align*}
\| \hat{u}-u^{\sharp}\|_{W^{1,\infty}(\Omega)} \leq \frac{\ep}{2}.
\end{align*}
On the other hand, since $u-u^\sharp$ vanishes on $x_2=\pm 1$, a trivial form of Hardy's inequality leads to
\begin{align*}
\left\| \frac{\hat{u}-u^{\sharp}}{\gamma} \right\|_{\infty} \leq  \|\nabla(u-u^\sharp)\|_\infty,
\end{align*}
so that eventually we have obtained $\Lambda u = \hat{u} \in \overline{B}_{\overline{\mathcal E}}(u^{\sharp},\ep)$.
\end{proof}
\begin{proof}[Proof of Lemma~\ref{LemLambdaCont}]
Consider $(u_n) \in \overline{B}_{\overline{\mathcal E}}(u^{\sharp},\ep)^\N$ converging to $u$ in $L^\infty(\Omega)$. Since $(\Lambda(u_n))$ takes its values in $\overline{B}_{\overline{\mathcal E}}(u^{\sharp},\ep)$ which is compact for the $L^\infty(\Omega)$ topology, it is sufficient to prove that $\Lambda(u)$ is the only possible cluster point (in $L^\infty(\Omega)$) for $(\Lambda(u_n))$ to prove the desired convergence. Since the Stokes equation with a right hand side in $L^2(\Omega)$ has a unique weak solution in $L^2(\Omega)$, it is sufficient to prove that the corresponding sequence of right hand sides of \eqref{NSL1} has a unique cluster point in the weak $L^2(\Omega)$ topology.
But $(u_n)$ converges strongly to $u$ in $L^\infty(\Omega)$, and since $\overline{B}_{\overline{\mathcal E}}(u^{\sharp},\ep)$ is compact for the $W^{1,\infty}(\Omega)$ weak-$\star$ topology, $(\nabla u_n)$ converges to $\nabla u$ weakly-$\star$ in $L^\infty(\Omega)$. By weak-strong convergence, it follows that $(u_n \cdot \nabla u_n)$ converges to $u \cdot \nabla u$ in $L^{\infty}(\Omega)$ weak-$\star$ and consequently weakly in $L^2(\Omega)$. \par
It now remains to treat the second term in the right hand side of \eqref{NSL1}. Since the sequence $(u_n)$ is bounded in $L^\infty(\Omega)$, one readily checks that $g_n$ is uniformly (with respect to $n$) compactly supported in velocity. Since $(g_n)$ is bounded in $L^\infty(\Omega\times\R^2)$ it is sufficient to prove the a.e. convergence of $(g_n)$ to $g$: the conclusion will then follow using the dominated convergence Theorem. This a.e. convergence is established using Lemma~\ref{propo:tau:n} in the following way. If $\tau_{x,v}=-\infty$, we have $(x,v)\notin A_K$ so that $g(x,v) = 0$ and point $(ii)$ of Lemma~\ref{propo:tau:n} implies that $(x,v)\notin A_K^n$ for $n$ large enough, so that $g_n(x,v)=0$ for $n$ large enough and $(g_n(x,v))_n$ indeed converges to $g(x,v)$. Else, if $\tau_{x,v} \neq -\infty$, we can restrict ourselves to the case $(X_{x,v},V_{x,v})\notin\Sigma^s$: indeed, an application of Sard's lemma as in Proposition~2.3 of \cite{bardos} allows to see that the set of all characteristic curves crossing $\Sigma^s$ at some point is Lebesgue negligible. Since $\tau_{x,v}\neq -\infty$ and  $(X_{x,v},V_{x,v})\notin\Sigma^s$, we have $\tau_{x,v}^n \neq -\infty$  for $n$ large enough (thanks to point $(i)$ of Lemma~\ref{propo:tau:n}). For such points, the formula \eqref{PropagationCaracteristiques} simply becomes 
\begin{equation*}
g(x,v)= \exp(-2\tau_{x,v}) \, \psi(X_{x,v},V_{x,v}),
\end{equation*}
were $\psi$ is extended by $0$ on $\Gamma^r \cup \Gamma^u \cup \Gamma^d$, and similarly
\begin{equation*}
g_n(x,v)= \exp(-2\tau_{x,v}^n) \, \psi(X_{x,v}^n,V_{x,v}^n),
\end{equation*}
so that the expected convergence follows from point $(i)$ of Lemma\,\ref{propo:tau:n}.
\end{proof}
The existence of $(\overline{u},\overline{f})$ follows then from Schauder's fixed point theorem. Now that the velocity field of such a fixed point actually belongs to $\mathscr{C}^{1}$ can be seen from \eqref{NSL1}-\eqref{NSL2} and elliptic regularity for the stationary Stokes equation \cite{frosties} (as already mentioned in Subsection \ref{subsec:const}). The estimate~\eqref{Petitesse} holds by construction.
This concludes the proof of the first statement of Theorem~\ref{Theo:Equilibria}.
\subsection{Lipschitz regularity}
\label{subsec:lipreg}
It remains to prove estimate \eqref{fbarlip}.
Recall the notation \eqref{DefAK} where $K$ is the support of $\psi$ and that since $K$ is compact, we have $K \subset \{v_1 \geq a\}$ for some $a>0$.
Thanks to Lemma~\ref{propo:tau}, we know that $g$ is continuous on $A_{K}$. Moreover one can see that $g$ is equal to $0$ on $\partial A_{K}$: indeed, let $(x,v) \in \partial A_K$. If $(x,v)$ belongs to $A_{K}$ then $(X_{x,v},V_{x,v}) \in \partial K$ so $g(x,v)=\psi(X_{x,v},V_{x,v})=0$ or else if $(x,v) \notin A_{K}$ then by definition $g(x,v)=0$. Since $g=0$ outside of $A_{K}$, we will only need to prove its Lipschitz regularity on the interior of $A_{K}$, a task that we shall fulfill by checking that $g$ is differentiable, with bounded derivatives.
The characteristics satisfy the following estimate 
\begin{equation*}
\|\nabla_{x,v} (X,V)\|_\infty \leq \exp(T\|\nabla_x \overline{u}\|_\infty +T).
\end{equation*}
On the other hand $\tau_{x,v}:=\tau_-(0,x,v)$ is defined locally by the equation $X_1(0,\tau_{x,v},x,v)=-L$.
Since $(x,v) \in A_K$ we also have $V_1(0,\tau_{x,v},x,v) \geq a>0$ so that by the implicit function theorem, $\tau_-$ is differentiable at $(x,v)$ and 
\begin{align*}
\| \nabla_{x,v} \tau\|_\infty \leq \frac{1}{a} \|\nabla_{x,v}(X,V)\|_\infty.
\end{align*} 
All in all we eventually get that $|\nabla_{x,v} g (x,v)| \leq C(T,\|\nabla_x u\|_\infty,\psi) \|\nabla_{x,v} \psi\|_\infty$, so that $g$ is indeed Lipschitz on $A_K$. \par
%
%
%
To conclude, we use point $(ii)$ of Lemma~\ref{propo:tau}: the set $A_K$ is at positive distance $\eta>0$ of the set ${B:=\{\tau_{x,v}=-\infty\}}$. Now we pick $(x,v)$ and $(z,w)$ in $\overline{\Omega} \times \R^2$, and discuss according to the following cases:
\begin{itemize}
\item[$\bullet$] {\bf Case 1.} When both $(x,v)$ and $(z,w)$ belong to $B$, then by definition of $g$ one has $g(x,v)=g(z,w)=0$.
\item[$\bullet$] {\bf Case 2.} When both belong to $A_K$, then as shown above, $|g(x,v)-g(z,w)|\leq L|(x,v)-(z,w)|$ where $L$ is the Lipschitz constant of $g$ on $A_K$.
\item[$\bullet$] {\bf Case 3.} When one of them belongs to $A_K$ and the other one belongs to $B$, then $|(x,v)-(z,w)|\geq \eta>0$, so that $|g(x,v)-g(z,w)| =  |g(x,v)| \leq \frac{\|g\|_\infty}{\eta}|(x,v)-(z,w)|$.
\end{itemize}

This  concludes the proof of the Lipschitz regularity of $\overline{f}$ and of the associated estimate~\eqref{fbarlip}
%
%
%
%
%
%
%

%
%
%
\section{Local stability}
\label{Sec:Stab}
We begin with the definition of local stability that is considered in this work.
\begin{defi}[Local stability with respect to a class of perturbations]
\label{defi:locsta}  
Let 
$$
\mathcal{F}  \subset \left\{ (u,f)  \in  L^2(\Omega)  \times L^\infty(\Omega\times\R^2), \, \div\, u= 0\right\}
$$
be a set of admissible perturbations.
A stationary solution $(\overline{u},\overline{f})$ of system \eqref{Kinetic}-\eqref{NavierStokes2} with boundary conditions \eqref{CBu1}-\eqref{CBf2} is called \emph{locally stable} with respect to perturbations in the class $\mathcal{F}$ if, for any fixed $R>0$ there exists $\ep_1,D,\lambda>0$ such that, for any $(u_0,f_0)\in L^2(\Omega)\times L^\infty(\Omega\times\R^2)$ such that $(u_0-\bar{u},f_0-\bar{f})\in \mathcal{F}$  and
$$
\|f_0-\overline{f}\|_\infty + \|u_0-\bar{u}\|_2 < \ep_1, 
$$
any Leray solution $(u,f)$ to  \eqref{Kinetic}-\eqref{NavierStokes2} with boundary conditions \eqref{CBu1}-\eqref{CBf2} with initial conditions $(u_0,f_0)$ satisfies 
$$
\forall t \in \R_+, \quad \| f(t) - \overline{f} \|_{L^2_{x,v}} +  \|u(t)-\overline{u}\|_2 \leq D e^{-\lambda t}.
$$
\end{defi}
The local stability that we prove in this Section concerns stationary states that are sufficiently close to a reference state $(u_p,0)$. Our precise statement is as follows.
%
%
%
%

\begin{thm} \label{thm-expo}
There exists a constant $K_\Omega >0$ such that the following hold. 
Consider two compact sets $K_{1} \subset \Gamma^l$ and $K_{2} \subset \Omega \times \R^{2}$.
Fix  $\psi \in \mathscr{C}^0_c(\Gamma^l)$ such that $\Supp(\psi) \subset K_{1}$. If $u_p$ is a Poiseuille flow satisfying, for some time $T>1$,
\begin{itemize}
\item the lateral EGC in time $T-1$ with respect to $K_{1}$;
\item the initial EGC  in time $T-1$  with respect  to $K_{2}$;
\item  $\|\nabla u_p\|_\infty \leq K_\Omega$;
\end{itemize} 
then, there exists a neighborhood $O$ of $(u_p,0)$ in $\mathcal{E}\times W^{1,\infty}(\Omega\times\R^2) $ such that any stationary solution 
of the Vlasov-Navier-Stokes system in the pipe belonging to $O$ is locally stable with respect to the class of perturbations 
$$
 \mathcal{F}:= \left\{ (u,f)  \in  L^2(\Omega)  \times L^\infty(\Omega\times\R^2), \, \div\, u= 0, \, \Supp f \subset K_{2} \right\}.
 $$
%
\end{thm}
\begin{rem}
The previous statement is of interest only if the assumptions can be matched and if there is a neighborhood $O$ that contains stationary solutions to our system. This is indeed the case, at least under appropriate conditions on $K_1$ and $K_2$, thanks to Lemma~\ref{poigeo} and Theorem~\ref{Theo:Equilibria}. In particular this requires $\psi$ to be Lipschitz with sufficiently small $W^{1,\infty}$ norm.
\end{rem}
\begin{rem}
It will be clear from the proof of this result that the exponential decay of the ${L^2_{x,v}}$ norm given in Definition \ref{defi:locsta} is also true for $L^p_{x,v}$ norms, for any finite value of $p$.
\end{rem}
\begin{rem}
\label{BC2}
With the other choice of boundary condition~\eqref{Poiautre} described in Remark~\ref{BC0}, because of Remark~\ref{BC1}, an analogue of Theorem~\ref{thm-expo} holds, with the following changes. No geometric assumption is needed on $K_1, K_2$ and $u_p$ and the local stability holds with respect to the larger class of perturbations 
$$
 \mathcal{F'}:= \left\{ (u,f)  \in  L^2(\Omega)  \times L^\infty(\Omega\times\R^2), \, \div\, u= 0, \, f \text{  is compactly supported}\right\}.
 $$
\end{rem}
As we will see, Theorem~\ref{thm-expo} is based on a kind of continuous induction argument, in which the solution is estimated by means of a delayed differential equation. This estimate itself relies on the (different) EGC satisfied by $u_p$: denoting $(\overline{u}, \overline{f})$ a stationary state in $O$, they are also satisfied by $\overline{u}$ and, in some sense, by the fluid velocity fields of Leray solutions whenever they are sufficiently close to $u_p$. %
\subsection{Local uniqueness}
Before starting the proof of the theorem, let us study one important consequence, that is a local uniqueness property of the stationary solution $(\overline{u},\overline{f})$ constructed in Theorem~\ref{Theo:Equilibria}.
Within the assumptions of Theorem~\ref{thm-expo}, it follows (see the proof of Lemma~\ref{suppfbar} below) that $u_p$ satisfies the EGC in some finite time  with respect to the support of $\overline{f}$, on $\R_+$. Also, consider a stationary solution  $(\mathfrak{u}, \mathfrak{f})$ close to $(\overline{u},\overline{f})$ for the topology of  $\mathcal{E}\times W^{1,\infty}(\Omega\times\R^2)$; likewise, it follows that $u_p$ satisfies the EGC in some finite time  with respect to the support of $\mathfrak{f}$, on $\R_+$.

We may therefore infer from Theorem \ref{thm-expo} the following uniqueness result.  
\begin{cor}\label{cor:uni}
The stationary solutions $(\overline{u},\overline{f})$ constructed in Theorem \ref{Theo:Equilibria} are locally unique in the following sense : there exists a $\mathcal{E}\times W^{1,\infty}(\Omega\times\R^2)$ neighboorhood of $(\overline{u},\overline{f})$ in which the latter is the only stationary solution. 
\end{cor}
The rest of the section is devoted to the proof of Theorem~\ref{thm-expo} and is divided into two parts. 

We fix once for all $K_1$, $K_2$, $\psi$, $u_p$ and $T$ satisfying the assumptions of Theorem~\ref{thm-expo}. We fix also $\delta$ to be the smallest of the three parameters introduced in Lemma~\ref{Lem:Gronwall}, Lemma~\ref{Lem:Gronwall'} and Lemma~\ref{Lem:Gronwall2} (respectively for $K_1$, $\textnormal{Supp}\,\overline{f}$ and $K_2$).
%
%
%
%
%
\subsection{Consequences of the exit geometric condition}
In this subsection, we consider a stationary state $(\overline{u}, \overline{f})$ belonging to some neighborhood $O$ of $(u_{p},0)$, and an initial condition $f_0$ such that $\Supp(f_0-\bar{f}) \subseteq K_2$.
Let us start by establishing a first important EGC automatically satisfied by $\overline{u}$ and $u_p$.
\begin{lem}
\label{suppfbar}
If the neighboord $O$ is small enough, both vector fields $\overline{u}$ and $u_p$ satisfy the internal lateral EGC in time $T$ with respect to the compact set $\Supp \overline{f}$.
\end{lem}

\begin{proof}[Proof of Lemma~\ref{suppfbar}]
This follows from the following facts. 
\par
Observe that all backwards characteristics associated to $u_p$ starting from any $(x,v) \in \Omega \times \R^2$ reach in finite time the incoming boundary $\Sigma^-$.
Since $u$ is close to $u_p$ in the $\mathcal{E}$ norm, this property also holds for the characteristics associated to $u$. 
Then, since $( \overline{u}, \overline{f})$ is a (stationary) solution to the Vlasov-Navier-Stokes system, it means that any element $(x,v) \in \Supp \overline{f}$ has to be issued from $K_1$. A first consequence is that $\Supp \overline{f}$ is a compact subset of $\overline \Omega \times \R^2$ (because $\overline{u}\in L^\infty(\Omega)$).
Next, we know that  by assumption, $u_p$ satisfies the lateral EGC  in time $T-1>0$ with respect to  $K_{1}$. By Lemma~\ref{Lem:Gronwall} and Remark~\ref{remT}, up to reducing the neighborhood $O$ appropriately, $\overline{u}$ satisfies the lateral EGC  in time $T-1/2$ with respect to $K_{1}$. It implies in particular that $\overline{u}$ satisfies the internal lateral EGC in time $T-1/2$ (and thus in time $T$) with respect to $\Supp \overline{f}$. Applying Lemma~\ref{Lem:Gronwall'}, up to reducing $O$, we infer that ${u}_p$ satisfies the internal lateral EGC in time $T$ with respect to $\Supp \overline{f}$.
\end{proof}
We fix $R>0$  such that $\Supp_v \overline f  + \Supp_v f_0  \subset B(0,R).$ 
Note that since in the sequel we will consider Leray solutions to the coupled system, we may manipulate vector fields $u$ with less than Lipschitz regularity. However Lemma~\ref{Lem:Gronwall} will be used on approximating vector fields $u_n$ which are smooth. 
%
\ \par 
Next, we obtain a differential inequality allowing to measure the distance to $u_p$ of a Leray solution. This is a key ingredient in the proof, for which we crucially use the different EGC. It is in this step that we use the assumption $\|\nabla u_p\|_\infty \leq K_\Omega$. More precisely, if $C_{\text{Po},\Omega}$ is the best Poincar\'e constant associated to the domain $\Omega$ we introduce
\begin{equation} \label{DefKOmega}
K_{\Omega} := 1/(2 C_{\text{Po},\Omega}^2).
\end{equation}

Now consider a Leray solution $(u,f)$ with initial condition $(u_0,f_0)$. We have the following
\begin{lem} \label{lem:nrj:delay}
Suppose $u \in L^{2}_{\rm loc}(\R_{+}^*; H^2(\Omega))\cap\mathscr{C}^0(\R_+^*;H^1(\Omega)) \cap \mathscr{C}^1(\R_+^*;L^2(\Omega))$ satisfies, for some $t^{\star} >T$ that
\begin{equation} \label{Hyp:nrj:delay}
\forall t \in [T,t^\star], \quad \|u - u_p \|_{L^1(t- T,t; L^\infty(\Omega))} < \delta.
\end{equation}
Introducing
\begin{equation} \label{Eq:DefHatR}
\hat{R}:=R + T \| u_p \|_{\infty} + 2 \delta,
\end{equation}
we have then the following estimates on $[T,t^\star]$, where we define
\begin{equation} \label{DefChapeaux}
\hat{f}:=f-\overline{f} \ \text{ and } \ \hat{u}:=u-\overline{u}.
\end{equation}
\begin{enumerate}
\item For $t\in [T,t^\star]$, we have $\Supp_v \hat{f}(t)\subset B(0,\hat{R})$ and furthermore for  $p \in [1,\infty]$,
\begin{equation} \label{normelp:delay}
\| \hat{f}(t) \|_{L^p_{x,v}} \leq (\pi R^2)^{1/p}  e^{2T}  \|\nabla_v \overline{f} \|_\infty \int_{t-T}^t  \| \hat{u}(s) \|_{L^p_x} \, ds.
\end{equation}
\item If $\|\nabla u_p \|_\infty \leq K_\Omega$ where $K_\Omega$ is defined by \eqref{DefKOmega} then for $t\in [T,t^\star]$, we have
$$
\frac{d}{dt} \| \hat{u}(t) \|_{L^2_x} + K_\Omega \| \hat{u}(t) \|_{L^2_x} \leq  \alpha \int_{t-T}^t \| \hat{u}(s) \|_{L^2_x} \, ds,
$$
where 
\begin{equation} \label{Eq:DefAlpha}
\alpha :=  R \pi e^{2T} \|\nabla_v \overline{f}\|_\infty \hat{R}(\|\overline{u}\|_\infty +\hat{R}).
\end{equation}
\end{enumerate}
\end{lem}
%
%
%
%
%
%
\begin{proof}[Proof of Lemma~\ref{lem:nrj:delay}]
The function $\hat{f}$ is a solution in the sense of Theorem~\ref{thm:vla} of the following Vlasov equation
\begin{equation*}
\partial_t\hat{f} + v\cdot \nabla_x \hat{f} +\div_v((u-v) \hat{f}) = -\hat{u} \cdot \nabla_v \overline{f},
\end{equation*}
with homogeneous boundary condition and initial condition $\hat{f}_0:=f_0-\overline{f}$. One can check directly that the conclusions of Theorem~\ref{thm:vla} still hold when adding an integrable source term in the right hand side of the Vlasov equation.  In particular, if $(u_n)$ is a sequence of approximations of $u$ in $L^1_{\textnormal{loc}}(\R_+;L^\infty(\Omega))$ by smooth functions (such a sequence exists since $u\in L^1_{\textnormal{loc}}(\R_+^*;H^2(\Omega))\hookrightarrow L^1_{\textnormal{loc}}(\R_+^*;\mathscr{C}^0(\overline{\Omega}))$), the stability property stated in Theorem~\ref{thm:vla} allows to prove that the corresponding solutions $(\hat{f}_n)$ of
\begin{equation*}
\partial_t\hat{f}_n + v\cdot \nabla_x \hat{f}_n + \div_v((u_n-v) \hat{f}_n) = - \hat{u}_n \cdot \nabla_v \overline{f},
\end{equation*}
satisfy $\hat{f}_n \rightarrow \hat{f}$ in $L^\infty_{\textnormal{loc}}(\R_+;L^p_{\textnormal{loc}}(\overline{\Omega}\times\R^2))$ as $n \rightarrow +\infty$ for any $p\in[1,\infty)$. 
Since $u_n$ is smooth, we may define its characteristics $(X^n, V^n)$ (following the notations of \eqref{Characteristics}) and we have 
\begin{multline*}
 \hat{f}_n(t,x,v) = e^{2t} \hat{f_0}(X^n(t,0,x,v),V^n(t,0,x,v)) \mathbf{1}_{ \tau_-^n(t,x,v)=0 } \\
- \int_0^t e^{2(t-s)} (\hat{u}_n \cdot \nabla_v \overline{f})(s,X^n(t,s,x,v),V^n(t,s,x,v)) \mathbf{1}_{ \tau_-^n(t,x,v)<s } \, ds,
\end{multline*}
where $\tau_-^n(t,x,v)$ is defined as in \eqref{DefTau-} with the corresponding vector field $u_{n}$. \par
When $n\rightarrow +\infty$, $u_n$ gets close to $u$ in $L^1_{\textnormal{loc}}(\R_+;L^\infty(\Omega))$, and we may infer from \eqref{Hyp:nrj:delay} that for $n$ large enough that for all $t\in [T,t^\star]$, 
$$
\| u_n - \overline{u} \|_{L^1(t- T,t; L^\infty(\Omega))} < \delta.
$$ 
Therefore, 
\begin{itemize}
\item according to Lemma~\ref{Lem:Gronwall2}, for $n$ large enough, $u_n$ satisfies the initial EGC in time $T$ with respect to $K_2$. We recall that we assume that $\Supp \hat{f}_0 \subset K_2$. This means that for all $(x,v) \in \Omega\times \R^2$ and all $t\geq T$, we have
$$
e^{2t} \hat{f_0}(X^n(t,0,x,v),V^n(t,0,x,v)) \mathbf{1}_{ \tau_-^n(t,x,v)=0 }=0.
$$
\item According to Lemma~\ref{suppfbar} and Lemma~\ref{Lem:Gronwall'}, for $n$ large enough, $u_n$ satisfies the internal lateral EGC in time $T$ with respect to $\Supp \overline f$ on $[0,t^\star-T]$. This means that for all $(x,v) \in \Omega\times \R^2$ and all $t\in [T,t^\star]$, we have
\begin{multline*}
 \int_0^t e^{2(t-s)} (\hat{u}_n \cdot \nabla_v \overline{f})(s,X^n(t,s,x,v),V^n(t,s,x,v)) \mathbf{1}_{ \tau_-^n(t,x,v)<s } \, ds \\
 =  \int_{t-T}^t e^{2(t-s)} (\hat{u}_n \cdot \nabla_v \overline{f})(s,X^n(t,s,x,v),V^n(t,s,x,v)) \mathbf{1}_{ \tau_-^n(t,x,v)<s } \, ds.
 \end{multline*}

\end{itemize}
 We thus have the simplified expression for any $t\in [T,t^\star]$:
\begin{equation}
\label{eqfn}
\hat{f}_n(t,x,v) = -\int_{t-T}^t e^{2(t-s)} (\hat{u}_n\cdot\nabla_v \overline{f})(s,X^n(t,s,x,v),V^n(t,s,x,v)) \mathbf{1}_{\tau_-^n(t,x,v)<s} \, ds,
\end{equation}
from which we first get
\begin{equation} \label{ineq:fnhlpinf}
\|\hat{f}_n(t)\|_\infty \leq e^{2T} \|\nabla_v \overline{f} \|_\infty \int_{t-T}^t  \|\hat{u}_n(s)\|_\infty  \, ds.
\end{equation}
For $p\in[1,\infty)$ we get from Minkowski's inequality
\begin{equation*}
\|\hat{f}_n(t)\|_{L^p_{x,v}} \leq  e^{2T} \int_{t-T}^t \|\hat{u}_n\cdot\nabla_v \overline{f}(s,X^n(t,s,x,v),V^n(t,s,x,v)) \mathbf{1}_{\tau_-^n(t,x,v)<s}\|_{L^p_{x,v}} \, ds.
\end{equation*}
Since $(X^n,V^n)$ is the flow associated to vector field $(t,x,v) \mapsto (v,u_n(t,x)-v)$ whose divergence in the phase space variables equals $-2$,  using the associated change of variables we are led to
\begin{align} \nonumber
\|\hat{f}_n(t)\|_{L^p_{x,v}} 
& \leq  e^{2T(1-\frac{1}{p})} \int_{t-T}^t \| \hat{u}_n(s) \cdot \nabla_v \overline{f} \|_{L^p_{x,v}} \, ds \\
\label{ineq:fnhlp}
& \leq  e^{2T} (\pi R^2)^{1/p} \| \nabla_{v} \overline{f} \|_{\infty} \int_{t-T}^t \|\hat{u}_n(s)\|_{L^p_{x}} \, ds.
\end{align}
On the other hand, from a view of~\eqref{eqfn}, for any $v \in \Supp_v\hat{f}_n(t)$ we have at least one $s\in[t-T,t]$ such that $V^n(t,s,x,v) \in B(0,R)$ from which,
thanks to the velocity equation in \eqref{Characteristics}, we deduce
$$
v \in B(0,R + \| u_n \|_{L^1([t-T,t];L^\infty(\Omega))} ),
$$
so that $\Supp_v \hat{f}_n \subset B(0,\hat{R})$.
Since $(\hat{f}_n)$ converges to $\hat{f}$ in $L^\infty_{\textnormal{loc}}(\R_+; L^p_{\textnormal{loc}}({\Omega}\times\R^2))$ for any $p \in [1,\infty)$, we  get $\Supp_v \hat{f} \subset B(0,\hat{R})$ and that the previous convergence in fact holds in $L^\infty_{\textnormal{loc}}(\R_+; L^p({\Omega}\times\R^2))$.
In particular, we may pass to the limit in \eqref{ineq:fnhlpinf} and \eqref{ineq:fnhlp} to first get \eqref{normelp:delay} for a.e. $t\geq T$ and then for all $t\geq T$ using that $\hat{f} \in \mathscr{C}^0(\R_+; L^\infty(\Omega\times\R^2)-w\star)$, so that point 1. is proven. \par
\ \par
Let us now focus on point 2.; the velocity field $\hat{u}$ solves the following equation
\begin{equation*}
\partial_{t} \hat{u} + (u \cdot \nabla) \hat{u} + (\hat{u} \cdot \nabla) \overline{u} - \Delta \hat{u} + \nabla \hat{p}
= j_{\hat{f}} - (m_0 f) \hat{u}  - (m_0 \hat{f}) \overline{u}.
\end{equation*}
Since $\hat{u}\in \mathscr{C}^0(\R_+^*;H^1(\Omega))\cap \mathscr{C}^1(\R_+^*;L^2(\Omega))$, we can multiply by $\hat{u}$ and integrate over $\Omega$, to get (using that $f$ is nonnegative)
\begin{equation*}
\frac{1}{2}\frac{d}{dt}\| \hat{u}(t) \|_2^2 + \| \nabla \hat{u}(t) \|_2^2
\leq \| \nabla \overline{u} \|_\infty \| \hat{u}(t) \|_2^2
+ \|\overline{u}\|_\infty (m_0 \hat{f}(t),\hat{u}(t))_{L^2(\Omega)}
+ (j_{\hat{f}}(t),\hat{u}(t))_{L^2(\Omega)}.
\end{equation*}
%
Since $\hat{u}$ vanishes on the boundary, we may use the Poincar\'e inequality, with optimal constant $C_{\textnormal{Po},\Omega}$.
\begin{equation*}
\frac{1}{2}\frac{d}{dt} \| \hat{u}(t) \|_2^2 + \frac{1}{C_{\textnormal{Po},\Omega}^{2}} \| \hat{u}(t) \|_2^2 
 \leq \| \nabla \overline{u} \|_\infty \| \hat{u}(t) \|_2^2 
+ \Big[ \| \overline{u} \|_\infty \| m_0 \hat{f}(t) \|_2 + \| j_{\hat{f}}(t)\|_2 \Big]  \| \hat{u}(t) \|_{L^2(\Omega)},
\end{equation*}
Using the bound $\| \nabla \overline{u} \|_\infty \leq K_\Omega$ (where we recall $K_\Omega$ was defined in~\eqref{DefKOmega}), we get
\begin{equation*}
\frac{d}{dt} \| \hat{u}(t) \|_2 + K_\Omega \| \hat{u}(t) \|_2 
\leq  (\| \overline{u} \|_\infty \| m_0 \hat{f}(t) \|_2 + \| j_{\hat{f}}(t) \|_2 ).
\end{equation*}
As we have proved above, $\Supp_v \hat{f}(t) \subseteq B(0,\hat{R})$ with $\hat{R}$ given in \eqref{Eq:DefHatR}, so that
\begin{equation*}
\|m_0 \hat{f}(t)\|_2 \leq \pi^{1/2} \hat{R} \| \hat{f}(t) \|_2,
\end{equation*}
and similarly
\begin{equation*}
\|j_{\hat{f}}(t)\|_2 \leq \pi^{1/2} \hat{R}^{2} \| \hat{f}(t) \|_2.
\end{equation*}
In the end we get
\begin{equation*}
\frac{d}{dt}\|\hat{u}(t)\|_2 + K_\Omega \| \hat{u}(t) \|_2 
\leq \pi^{1/2} \hat{R} (\|\overline{u}\|_\infty + \hat{R}) \| \hat{f}(t) \|_2, 
\end{equation*}  
and the conclusion follows using case $p=2$ of point 1. 
\end{proof}
%
%
%
%
%
%
%
%
%
\subsection{The continuity argument and a delayed Gronwall inequality}
We are now in position to prove Theorem~\ref{thm-expo}. \par
\begin{proof}[Proof of Theorem~\ref{thm-expo}]
The proof is based on the closeness  of $(\overline{u}, \overline{f})$ to the reference state $(u_p,0)$ in $\mathcal{E} \times W^{1,\infty}(\Omega\times\R^2)$, but also on the one of $(u_0,f_0)$ to $(\overline{u},\overline{f})$ in $L^2(\Omega) \times L^\infty(\Omega\times\R^2)$. To measure this closeness we introduce two parameters $\ep_0$ and $\ep_1$ such that 
\begin{align}
\label{eq:ep0} \|\overline{f}\|_{W^{1,\infty}(\Omega\times\R^2)} + \|\overline{u}-u_p\|_{\mathcal{E}} &\leq \ep_0,\\
\label{eq:ep1} \|f_0-\overline{f}\|_{\infty} + \|u_0-\overline{u}\|_{2} &\leq \ep_1.
\end{align}
Once these two parameters are fixed, the neighborhood $O$ will then simply be the ball of center $(u_p,0)$ and radius $\ep_0$ in $\mathcal{E} \times W^{1,\infty}(\Omega\times\R^2)$. 
The idea is to use Lemma~\ref{lem:nrj:delay}
to prove the stability. We therefore have to ensure that properties like~\eqref{Hyp:nrj:delay} are satisfied on $\R_+$, that is to say to ensure that the EGC properties are propagated. Before heading to the main matter of the proof, let us treat the case of small times
\begin{lem}\label{lem:nrj:delay:small}
There is $\ep_1>0$ small enough so that
\begin{equation} \label{Hyp:nrj:delay:small}
\forall t \in [T,3T/2], \quad \|u - u_p \|_{L^1(t- T,t; L^\infty(\Omega))} < \delta.
\end{equation}
\end{lem}
\begin{proof}[Proof of Lemma \ref{lem:nrj:delay:small}] 
We simply use the fourth item in Proposition~\ref{propo:nrj-def} to control $\|u-u_p\|_{L^1(0,3T/2;L^\infty(\Omega))}$.
\end{proof}
\par
\bigskip
Define
\begin{align*}
t^\star := \sup \{ t \geq \tfrac{3T}{2} \, : \, \| u - u_p \|_{L^1(s-T,s;L^\infty(\Omega))} < \delta,
\, \forall s\in[\tfrac{3T}{2},t] \}.
\end{align*}
%
%
%
We still fix $R>0$ such that $\Supp_v f_0 + \Supp_v \overline{f} \subset B(0,R)$. We will use again in what follows the notations 
\begin{align*}
\widetilde{u}&= u - u_p, \\
\hat{u}&= u- \overline{u}, \\
\hat{f}&= f- \overline{f}.
\end{align*}
Since $\overline{f}=\psi$ on $\Gamma^l$, the regularity estimate \eqref{ineq:ell} implies
\begin{equation*}
\int_{T/2}^{3T} \| \widetilde{u}(s) \|_\infty \, ds
\leq C_{\Omega, T, R, u_p} (\| \overline{f} \|_\infty, \| f_0 \|_\infty, M_4 f_0, \| \widetilde{u}(0) \|_2).
\end{equation*}
For $\ep_0$ and $\ep_1$ small enough in  \eqref{eq:ep0} - \eqref{eq:ep1}, we have $t^\star \geq 3T$: in particular $[\frac{3T}{2},t^\star)$ is nonempty. Now our goal is to prove that $t^\star =+ \infty$ so that we have the corresponding estimates for all positive times. \par
\ \par
To that purpose we assume temporarily that $t^\star < +\infty$ and invoke once more the regularity estimate \eqref{ineq:ell} to write,
for any $s\in [t^\star, t^\star+\frac{T}{2}]$ 
\begin{equation} \label{EstTU}
\| \widetilde{u} \|_{L^1(s-T,s; L^\infty(\Omega))}  \leq C_{\Omega,T,R,u_p} ( \| \overline{f} \|_\infty, \|f(s-\tfrac{3T}{2})\|_\infty, 
M_4 f(s-\tfrac{3T}{2}), \| \widetilde{u}(s - \tfrac{3T}{2}) \|_2 ),
\end{equation}
with $C_{\Omega,T,R,u_p}$ vanishing continuously at $0$. To estimate the second and third arguments in the right hand side, we use Lemma~\ref{lem:nrj:delay}. Note that we can apply this result because of the instantaneous regularization of Leray solutions: $u \in L^{2}_{\rm loc}(\R_{+}^*; H^2(\Omega))\cap\mathscr{C}^0(\R_+^*;H^1(\Omega)) \cap \mathscr{C}^1(\R_+^*;L^2(\Omega))$ (see Proposition~\ref{propo:nrj-def}). We obtain from the first part of Lemma \ref{lem:nrj:delay} that for all $t$ in $[\frac{3T}{2}, t^\star)$:
\begin{equation*}
\| f(t) \|_\infty \leq \| \overline{f} \|_\infty + \| \hat{f} \|_\infty 
\leq \| \overline{f} \|_\infty + e^{2T} \| \nabla_v \overline{f} \|_\infty \delta.
\end{equation*}
We have also $\textnormal{Supp}_v \hat{f}(t) \subset B(0,\hat{R})$ with $\hat{R}=R +\delta+T\|u_p\|_\infty$, thus 
\begin{align*}
 M_4 f(t) &\leq  M_4 \overline{f} + M_4 |\hat{f}| \\
 &\leq \pi R^6 \|\overline{f}\|_\infty + \pi \hat{R}^6 \|\hat{f}\|_{\infty} \\
 &\leq \pi R^6 \|\overline{f}\|_\infty + e^{2T} \pi \hat{R}^6 \|\nabla_v \overline{f}\|_\infty \delta,
\end{align*}
so that we actually established  for $t\in [\frac{3T}{2},t^\star)$ 
\begin{equation} \label{ineq:bootstrap}
\|f(t)\|_\infty + M_4 f(t) \leq C_{\Omega,T,R,u_p,\delta}(\|\overline{f}\|_{W^{1,\infty}}).
\end{equation} 
Going back to \eqref{EstTU}, for any $s\in [t^\star, t^\star+\frac{T}{2}]$ (so that $s-\frac{3T}{2} \in [\frac{3T}{2},t^\star)$), we may use \eqref{ineq:bootstrap} to write
\begin{equation} \label{eq:super}
\| \widetilde{u} \|_{L^1(s-T,s; L^\infty(\Omega))}
\leq C_{\Omega,T,R,u_p,\delta} (\|\overline{f}\|_{W^{1,\infty}}, \| \widetilde{u}(s - \tfrac{3T}{2}) \|_2).
\end{equation}
Since $C_{\Omega,T,R,u_p,\delta}$ vanishes continuously at $(0,0)$ we assume from now on that $ \|\overline{f}\|_{W^{1,\infty}}$ is small enough (reducing again $\varepsilon_{0}$ in \eqref{eq:ep0} if necessary) so that there exists $\tilde{\varepsilon} >0$ such that 
\begin{equation} \label{eq:conti}
0 \leq A < \tilde{\varepsilon} \Longrightarrow C_{\Omega,T,R,u_p,\delta} ( \|\overline{f}\|_{W^{1,\infty}}, A ) < \delta.
\end{equation}
Now it only remains to check that $\| \widetilde{u}(s-\tfrac{3T}{2}) \|_2$ is indeed less or equal to  $\tilde{\varepsilon}$ for suitable data.
To this purpose, we use the second part of Lemma~\ref{lem:nrj:delay} to write, for all $t \in [\frac{3T}{2},t^{\star}]$,
\begin{equation} \label{IneqDiff}
\frac{d}{dt} \| \hat{u}(t) \|_{L^2_x} + K_\Omega \| \hat{u}(t) \|_{L^2_x} \leq  \alpha \int_{t-T}^t \| \hat{u}(s) \|_{L^2_x} \, ds,
\end{equation}
where $\alpha$ is given by \eqref{Eq:DefAlpha}, that is $\alpha =  R \pi e^{2T} \|\nabla_v \overline{f}\|_\infty \hat{R}(\|\overline{u}\|_\infty +\hat{R})$.
Modifying again $\varepsilon_{0}$ in \eqref{eq:ep0} if necessary, we impose that $\| \overline{f} \|_{W^{1,\infty}}$ is small enough in order to ensure 
\begin{equation} \label{ineq:alpha}
\alpha < \frac{K_\Omega}{T}.
\end{equation}
Now we can use the following  Gronwall type result, whose proof we temporarily delay.
\begin{lem} \label{retarded}
Let $\kappa,\alpha, T>0$ such that $\alpha<\kappa/T$. Assume that $y\in\mathscr{C}^1(\R_+^*)\cap\mathscr{C}^0(\R_+)$ satisfies for $t \in (\frac{3T}{2},t^\star)$ the delayed inequality
\begin{equation} \label{ineq:delay}
y'(t) + \kappa y(t) \leq \alpha \int_{t-T}^t y(s) \, ds.
\end{equation}
Then there exists $\lambda>0$ such that for $t \in [0,t^\star)$,
\begin{equation} \label{ineq:delay:exp}
y(t) \leq H_{\lambda,T} e^{-\lambda t},
\end{equation}
where $H_{\lambda,T}:= \sup_{t \in [0,\frac{3T}{2}]} |y(t)| e^{\lambda t}$. Moreover, $\lambda$ can be chosen as a non-increasing function of $\alpha$.
\end{lem}
Thanks to~\eqref{IneqDiff} and \eqref{ineq:alpha}, we can apply Lemma~\ref{retarded} to $y(t)= \| \hat{u}(t) \|_{2}$. 
This implies the existence of $\lambda >0$ such that
\begin{align} \label{ineq:lambda}
\|\hat{u}(t)\|_2 \leq H_{\lambda,T} e^{-\lambda t},
\end{align}
where here 
\begin{equation} \label{HHere}
H_{\lambda,T}= \sup_{t \in [0,\frac{3T}{2}]} \| \hat{u}(t) \|_2 e^{\lambda t}. 	
\end{equation}
Note that $\lambda$ may depend on $\| \nabla_v \overline{f}\|_\infty$, but due to the monotonicity with respect to $\alpha$ in Lemma~\ref{retarded}, one may reduce again the maximal size of $\| \nabla_v \overline{f}\|_\infty$ later,  but without further modifying $\lambda$. \par
Concerning $H_{\lambda,T}$, thanks to the inequality \eqref{ineq:nrj-def} (since $\overline{f} = \psi$ on $\Gamma^l$), we estimate
\begin{equation*}
\sup_{t \in [0,\frac{3T}{2}]}\|\hat{u}(t)\|_2 \leq \|\overline{u}-u_p\|_2+ C_{\Omega,T,R,u_p}(\|\overline{f}\|_\infty,\|f_0\|_\infty,M_2f_0,\|\widetilde{u}(0)\|_2).
\end{equation*}
Then, using \eqref{ineq:lambda} and \eqref{HHere} we deduce that for $t$ in the whole interval $[0,t^\star)$, one has
\begin{equation*}
\| \hat{u}(t) \|_2 \leq C_{\Omega,T,R,u_p,\lambda} (\|\overline{u}-u_p\|_2, \|\overline{f}\|_\infty, \|f_0\|_\infty, M_2f_0, \|\widetilde{u}(0)\|_2) e^{-\lambda t}.
\end{equation*}
Now, by choosing again $\varepsilon_{0}<\tilde{\ep}/(4\sqrt{L})$ sufficiently small in \eqref{eq:ep0}, we can require that $\| \overline{u} - u_p \|_{L^2(\Omega))}$ and $\|\overline{f}\|_\infty$ are small enough to ensure the existence of $\tilde{\ep}'>0$ such that
$$
0 \leq B_1,B_2,B_3 < \tilde{\varepsilon}' \Longrightarrow C_{\Omega,T,R,u_p,\lambda} ( \|\overline{u}-u_p\|_2, \|\overline{f}\|_\infty, B_1,B_2,B_3 ) < \tilde{\varepsilon}/2.
$$
Therefore, taking $\ep_1$ revelantly in \eqref{eq:ep1} we can assume moreover that $\| f_0 \|_\infty$, $M_2 f_0$ and $\| \widetilde{u}(0) \|_2$ are also small enough to get
\begin{align*}
\| \tilde{u}(t^\star-\tfrac{3T}{2}) \|_2  &< \| \hat{u}(t^\star-\tfrac{3T}{2}) \|_2 + \|\bar{u}-u_p\|_2 \\
&< \tilde{\ep}/2 + \ep_0\\
&< \tilde{\ep}.
\end{align*} 
Recalling~\eqref{eq:super} and \eqref{eq:conti}, we deduce that
$$
\|u-u_p\|_{L^1(t^\star-T,t^\star;L^\infty(\Omega))} < \delta.
$$
However, according to the definition of $t^\star$, by continuity we have
$$
\|u-u_p\|_{L^1(t^\star-T,t^\star;L^\infty(\Omega))} =\delta,
$$
which is a contradiction. As a result, we necessarily have $t^\star=+\infty$ and the estimate \eqref{ineq:lambda} holds on $\R_+$. \par
%

The exponential convergences for $\hat{u}$ and  $\hat{f}$ finally follow from another use of Lemma~\ref{lem:nrj:delay} and Lemma~\ref{retarded}.
\end{proof}
\ \par
Now there only remains to prove Lemma~\ref{retarded}.
\begin{proof}[Proof of Lemma~\ref{retarded}]
First notice that if $y$ is negative on $[0,\frac{3T}{2}]$, then it remains negative on $[0,t^\star)$.
Indeed in that case, if $t^-:=\sup\,\{t\in [0,t^\star)\,:\, y_{|[0,t]} < 0\} < t^\star$, then by continuity, we necessarily have
$$
y(t^-) =0 \text{ and } y'(t^-)\geq 0.
$$
But from a view of \eqref{ineq:delay} at $t= t^-$, we infer that
$$
y'(t^-) + \kappa y(t^-) <0,
$$
which is a contradiction. \par
%
%
%
For \eqref{ineq:delay:exp}, it is then sufficient to prove the existence of $\lambda>0$ such that $z:t\mapsto e^{-\lambda t}$ is solution of
\begin{align} \label{eq:delay:exp}
z'(t) + \kappa z(t) = \alpha \int_{t-T}^t z(s) d s.
\end{align}
Indeed, if such $\lambda$ exists, for any $ \gamma > 1$, the function $t \mapsto y- \gamma H_{\lambda,T}e^{-\lambda t}$ satisfies \eqref{ineq:delay} and clearly $t \mapsto y- \gamma H_{\lambda,T}e^{-\lambda t}< 0$ on $[0,\frac{3T}{2}]$, so that the previous remark implies  \eqref{ineq:delay:exp}. \par
Now going back to \eqref{eq:delay:exp}, we see that it comes down to find  $\lambda>0$ such that
\begin{align*}
-\lambda +\kappa =\frac{\alpha}{\lambda}(e^{\lambda T} -1),
\end{align*}
or equivalently to find a positive root for the function $\varphi:\lambda\mapsto \lambda^2 -\lambda \kappa + \alpha(e^{\lambda T}-1)$. We have $\varphi(0)=0$ and $\lim_{+\infty} \varphi = +\infty$. The assumption $\alpha< \kappa /T$ implies $\varphi'(0)  <0$, from which we deduce by continuity the existence of  $\lambda>0$ such that $\varphi(\lambda)=0$. Using the convexity of $\varphi$ and $\varphi(\beta) >0$ for $\beta > \alpha$, we see that this $\lambda$ is unique and one deduces the monotonicity of $\lambda$ with respect to $\alpha$.

\end{proof}

\section{Appendix}
\label{Sec:Appendix}
In this Appendix we gather several technical results used in the proofs. 
\subsection{Regularity estimates for Navier-Stokes system in a rectangle}
\begin{thm}\label{th:exNSpol}
Fix $\Omega=(-L,L)\times(-1,1)$ a rectangle, and consider $u_0\in H^1_0(\Omega)$ and $F\in L^2_{\textnormal{loc}}(\R_+;L^2(\Omega))$. There exists a unique solution $(\widetilde{u},\widetilde{p})$ of 
\begin{align*}
  \partial_t \widetilde{u} + (\widetilde{u}\cdot \nabla ) \widetilde{u} - \Delta \widetilde{u} + \nabla \widetilde{p} &= F,\\
\div \,\widetilde{u} &=0,\\
u_{|t=0}&= u_0.
\end{align*}
such that $\widetilde{u}\in \mathscr{C}^0(\R_+;H^1_0(\Omega))\cap L^2_{\textnormal{loc}}(\R_+;H^2(\Omega))$, $\partial_t \widetilde{u} \in L^2_{\textnormal{loc}}(\R_+;L^2(\Omega))$ and $\widetilde{p}\in L^2_{\textnormal{loc}}(\R_+;H^1(\Omega))$.
For any $T>0$ this solution satisfies furthermore for $t \in [0,T]$
\begin{align}
\label{ineq:nrjnspol}
\frac{1}{2} \|\widetilde{u}(t)\|_{L^2(\Omega)}^2 + \int_0^t \|\nabla \widetilde{u}(s)\|_{L^2(\Omega)}^2 \, ds
&= \frac{1}{2}\|u_0\|_{L^2(\Omega)}^2 + \int_0^t \langle F(s),\widetilde{u}(s)\rangle_{L^2(\Omega)} \,ds, \\
\label{ineq:ellnspol}
\| \widetilde{u}(t)\|_{H^1(\Omega)}^2  + \int_0^T \|\widetilde{u}(s)\|_{H^2(\Omega)}^2 \, ds +\int_0^T \|\partial_t \widetilde{u}(s)\|_{L^2(\Omega)}^2 \, ds
& \leq C_{T,\Omega}\Big(\|u_0\|_{H^1(\Omega)},\|F\|_{L^2((0,T)\times\Omega)}\Big),
\end{align}
where $C_{T,\Omega}(\cdot,\cdot)$ is a positive continuous function vanishing at $(0,0)$ and nondecreasing in with respect to each of its variable.
\end{thm}
\begin{proof}[Proof of Theorem~\ref{th:exNSpol}]
 Theorem V.2.1 (p.370) of \cite{boyer-fabrie}  gives exactly the above result (even though \eqref{ineq:ellnspol} is not explicitly written), but therein the domain $\Omega$ is assumed to be $\mathscr{C}^{1,1}$. Reading carefully the proof of \cite{boyer-fabrie}, this regularity assumption is mandatory only to ensure elliptic regularity for the Stokes operator. Elliptic regularity for the Stokes operator in a convex polygon is established in \cite{frosties}. 
\end{proof}
In the same way, using again \cite{frosties}, adapting the proof of Theorem V.2.12 (p.390) of \cite{boyer-fabrie} we have  the following regularization over  time for the Leray solutions.
\begin{thm}\label{thm:parabo}
Fix $\Omega=(-L,L)\times(-1,1)$ a rectangle, and consider $u_0\in L^2(\Omega)$ and $F\in L^2_{\textnormal{loc}}(\R_+;L^2(\Omega))$. The Leray solution $\widetilde{u}\in \mathscr{C}^0(\R_+;L^2(\Omega))\cap L^2_{\textnormal{loc}}(\R_+;H^1_0(\Omega))$ of 
\begin{align*}
\partial_t \widetilde{u} + (\widetilde{u}\cdot \nabla ) \widetilde{u} - \Delta \widetilde{u} + \nabla \widetilde{p} &= F,\\
\div\, \widetilde{u} &=0,
\end{align*}
satisfies $\widetilde{u}\in L^2_{\textnormal{loc}}(\R_+^*;H^2(\Omega))\cap \mathscr{C}^0(\R_+^*; H^1_0(\Omega))\cap \mathscr{C}^1(\R_+^* ; L^2(\Omega)) $. More precisely, the following estimate holds for any $T>0$, 
\begin{align} \label{ineq:thm:parabo}
\sup_{s\in[0,T]}s \| \widetilde{u}(s)\|_{H^1(\Omega)}^2  + \int_0^{T} s\|\widetilde{u}(s)\|_{H^2(\Omega)}^2 \, ds +\int_0^{T} s\|\partial_t \widetilde{u}(s)\|_{L^2(\Omega)}^2 \, ds
&\leq C_{T,\Omega}\Big(\|u_0\|_{L^2(\Omega)},\|F\|_{L^2((0,T)\times\Omega)}\Big),
\end{align}
where $C_{T,\Omega}(\cdot,\cdot)$ is a positive continuous function vanishing at $(0,0)$ and nondecreasing in each of its variable.
\end{thm}
A time translation argument leads then to the following corollary.
\begin{cor}\label{coro:parabo}
Fix $T>0$. Under the assumptions of Theorem~\ref{thm:parabo}, the solution $\widetilde{u}$ satisfies the following estimate for any $a\geq \frac{T}{2}$ and $b=a+T$
\begin{align} \label{ineq:coro:parabo}
\sup_{s\in[a,b]} \| \widetilde{u}(s)\|_{H^1(\Omega)}^2  + \int_{a}^{b} \|\widetilde{u}(s)\|_{H^2(\Omega)}^2 \, ds
&\leq C_{T,\Omega}\Big(\|\widetilde{u}(a-\tfrac{T}{2})\|_{L^2(\Omega)},\|F\|_{L^2((a-\tfrac{T}{2},b)\times\Omega)}\Big),
\end{align}
where $C_{T,\Omega}(\cdot,\cdot)$ is a positive continuous function vanishing at $(0,0)$ and nondecreasing with respect to each of its variable.
\end{cor}
\begin{proof}[Proof of Corollary~\ref{coro:parabo}]
Note that $s\mapsto \widetilde{u}(s+a-\frac{T}{2})$ is a Leray solution with initial condition $\widetilde{u}(a-\frac{T}{2})\in L^2(\Omega)$. Theorem~\ref{thm:parabo} hence implies, denoting $c:=a-\frac{T}{2}$
\begin{align}
\sup_{s\in[c,b]}(s-c) \| \widetilde{u}(s)\|_{H^1(\Omega)}^2  + \int_{c}^{b} (s-c)\|\widetilde{u}(s)\|_{H^2(\Omega)}^2 \, ds
&\leq C_{T,\Omega}\Big(\|\widetilde{u}(c)\|_{L^2(\Omega)},\|F\|_{L^2((c,b)\times\Omega)}\Big),
\end{align}
from which we infer \eqref{ineq:coro:parabo}.
\end{proof}
\subsection{Br\'ezis-Gallou\"et inequality}
In \cite{bg} the following inequality is proved in the case of a smooth domain $\Omega$. Actually, the proof (see Lemma 2 of \cite{bg}) uses the regularity of $\Omega$ only through the extension operator $H^1(\Omega)\rightarrow  H^1(\R^2)$ so that we infer the following Lemma in our setting 
\begin{lem}[Br\'ezis-Gallou\"et]
\label{lem:bg}
Fix $\Omega=(-L,L)\times(-1,1)$ a rectangle. There exists $C_\Omega >0$ such that for any $u\in H^2(\Omega)$, there holds
\begin{align}
\label{ineq:bg} \|u\|_{L^\infty(\Omega)} \leq C_\Omega \|u\|_{H^1(\Omega)} \left[
1+\sqrt{\log\left(1+\frac{\|u\|_{H^2(\Omega)}}{\|u\|_{H^1(\Omega)}}\right)}\,\right].
\end{align}
\end{lem}
\subsection{Boundary value problem for the kinetic equation} 
The following result is a straightforward consequence of Theorem VI.1.6 (p.423) of  \cite{boyer-fabrie}, which is an adaptation of the celebrated theory of DiPerna-Lions \cite{DPL} to the case of transport equations set in domains with boundary: note that even though the phase space domain $\Omega\times \R^2$ is not bounded, the proof of \cite{boyer-fabrie} applies \emph{verbatim}. 

\begin{thm} \label{thm:vla}
Fix $\chi \in \mathscr{C}^\infty(\R)$ such that $|\chi(z)|\leq |z|$ and $\chi'\in L^\infty(\R)$. 
Take $u \in L^1_{\textnormal{loc}}(\R_+;W^{1,1}(\Omega))$, $\psi\in L^\infty(\R_+\times\Gamma^l )$ (with compact support in $v$) and $f_0\in L^1\cap L^\infty(\Omega\times\R^2)$. 
Consider the kinetic equation
\begin{align*}
\partial_t f + v \cdot \nabla_x f + \div_v(\chi(u-v)f)=0,
\end{align*}
with boundary conditions \eqref{CBf1}--\eqref{CBf2} and initial data $f_0$, solutions being understood in the sense of Definition~\ref{def:weak}.
Then we have 
\begin{itemize}
\item[$\bullet$] \emph{Well-posedness:} There exists a unique $f\in L^\infty_{\textnormal{loc}}(\R_+;L^1 \cap L^\infty(\Omega\times\R^2))$ solution of the previous Cauchy boundary value problem, satisfying furthermore $f\in\mathscr{C}^0(\R_+;L^p_{\textnormal{loc}}(\overline{\Omega}\times\R^2))$ for all $p<\infty$.
\item[$\bullet$] \emph{Stability:} If 
$$(u_n) \rightarrow u \text{ in  } L_{\textnormal{loc}}^1(\R_+;L^1(\Omega)),$$ 
the corresponding sequence $(f_n)$ satisfies for all $p<\infty$,
$$
(f_n) \rightarrow f \text{ in } L^p_{\textnormal{loc}}(\R_+;L^p_{\textnormal{loc}}(\overline{\Omega}\times\R^2)),$$ where $f$ is the solution corresponding  to $u$. 
\item[$\bullet$] \emph{Maximum principle:} We have for all $T>0$, 
$$\sup_{0\leq t\leq T} \|f(t)\|_{\infty} \leq e^{2T}(\|f_0\|_\infty  + \|\psi\|_\infty).$$ 
Furthermore if $f_0\geq 0$ and $\psi\geq 0$ then $f\geq 0$.
\item[$\bullet$] \emph{Moments estimate:} One has 
\begin{align*}
M_0 f (t) = M_0 f_0 - \int_0^t \int_{\Gamma^l} \psi(s,x,v)v\cdot n(x) \,dv\,dx\,ds,
\end{align*}
and furthermore if  $M_\alpha f_0<\infty$ for some $\alpha>0$, then for all $t\in \R_+$  
\begin{align*}
M_\alpha f(t) &= \alpha \int_0^t \int_{\Omega\times\R^2} |v|^{\alpha-2} v\cdot \chi(u-v) f  \, dv \, dx \, ds + M_\alpha f_0 - \int_0^t \int_{\Gamma^l} \psi(s,x,v)|v|^\alpha v\cdot n(x)  \, dv \, dx \, ds.
\end{align*}
\end{itemize}
\end{thm}
\subsection{Interpolation}
We now provide some technical results regarding the interpolation of moments in velocity for solutions of Vlasov equations.
\begin{lem} \label{interpol}
Let $f:\R\times \Omega \times\R^2\rightarrow\R_+$. Recall the notation
\begin{align*}
m_\alpha f (s,x) &:= \int_{\R^2} f(s,x,v)|v|^\alpha  \, dv,\\
M_\alpha f(s) &:= \int_{\Omega} m_\alpha f(s,x) \, dx.
\end{align*}
One has for all $\beta \leq \gamma $,
\begin{align*}
\| m_\beta f(t) \|_{\frac{\gamma+2}{\beta+2}} \leq C_{\beta,\gamma} \|f(t)\|_\infty^{\frac{\gamma-\beta}{\gamma+2}}M_\gamma f(t)^{\frac{2+\beta}{\gamma+2}}.
\end{align*}
\end{lem}

\begin{cor}\label{coro:vla}
Under the assumptions of Theorem~\ref{thm:vla} with $\chi(z)=z$, if furthermore we assume that $u\in L^2_{\textnormal{loc}}(\R_+;L^6(\Omega))$, $\textnormal{Supp}_v(\psi) \subset B(0,R)$ and $M_4 f_0 < \infty$ then the following estimate holds on any interval $[0,T]$ 
\begin{align*}
M_4f(t) \leq C_{\Omega,T,R}(\|\psi\|_\infty, M_4 f_0,\|f_0\|_\infty) D_T(\|u\|_{L^2(0,T;L^6(\Omega))}),
\end{align*}
where $C_{\Omega,T,R}$ and $D_T$ are two continuous nonnegative nondecreasing functions with respect to each of their arguments, the first one furthermore vanishing at $0$.
\end{cor}
\begin{proof}[Proof of Corollary~\ref{coro:vla}]
We use the moments estimate of Theorem~\ref{thm:vla} to write 
\begin{align*}
M_4 f(t)  + 4 \int_0^t M_4 f(s) \, ds  \leq 4\int_0^t \int_{\Omega} |m_{3} f| |u|(s,x) \, dx \, ds + M_4 f_0 - \int_0^t \int_{\Gamma^l} \psi(s,x)|v|^4 v\cdot n(x) \, dx  \, dv ,
\end{align*}
and using $\text{Supp}_v(\psi)\subset B(0,R)$ we get by H\"older's inequality 
\begin{align*}
M_4 f(t)  + 4 \int_0^t M_4 f(s) \, ds  &\leq C_{\Omega,T,R}(\|\psi\|_\infty) + M_4 f_0 + 4\int_0^t \| m_{3} f(s)\|_{6/5} \|u(s)\|_{6} \, ds. 
\end{align*}
Thanks to the interpolation Lemma~\ref{interpol} with $\beta = 3$ and $\gamma=4$, we get using the Cauchy-Schwarz inequality
\begin{align*}
M_4 f(t)  &+ 4 \int_0^t M_4 f(s) \, ds \\
 &\leq C_{\Omega,T,R}(\|\psi\|_\infty) + M_4 f_0 +  C\|f\|_\infty^{1/6}\left(\int_0^t M_4 f(s)^{5/3} \, ds \right)^{1/2} \left(\int_0^t \|u(s)\|_{6}^2 \, ds\right)^{1/2},
\end{align*}  
Rising to the power $2$ the previous inequality, we get (changing the constants if necessary) thanks to H\"older's inequality for $t\in[0,T]$
\begin{align*}
M_4 f(t)^2 &\leq C_{\Omega,T,R}(M_4 f_0,\|\psi\|_\infty) + C\|f\|_\infty^{1/3} \|u\|_{L^2(0,T;L^6(\Omega))}^2 \int_0^t M_4 f(s)^{5/3} ds \\
&\leq C_{\Omega,T,R}(M_4 f_0, \|\psi\|_\infty) + C\|f\|_\infty^{1/3} \|u\|_{L^2(0,T;L^6(\Omega))}^2  T^{1/6} \left(\int_0^t M_4 f(s)^{2} ds\right)^{5/6}.
\end{align*}
Using  $|ab|\leq \frac{1}{6} |a|^6 + \frac{5}{6} |b|^{6/5}$ and $\|f\|_\infty \leq C_T(\|f_0\|_\infty,\|\psi\|_\infty)$ (see the maximum principle of Theorem~\ref{thm:vla}) we infer
\begin{align*}
M_4 f(t)^2 \leq C_{\Omega,T,R}(M_4 f_0,\|f_0\|_\infty, \|\psi\|_\infty)+ \|u\|_{L^2(0,T;L^6(\Omega))}^{12/5}\int_0^t M_4f(s)^2 ds,
\end{align*}
and the conclusion follows thanks to Gronwall's lemma.
\end{proof}

\bibliographystyle{plain}
\bibliography{Poicouette}

\end{document}